\documentclass[11pt,a4paper]{amsart}
\usepackage{amsmath, amssymb}
\usepackage{amsfonts}
\usepackage[arrow,matrix,curve,cmtip,ps]{xy}
\usepackage{amsthm}
\usepackage{amscd}
\usepackage[pdftex,colorlinks,linkcolor=magenta,citecolor=cyan]{hyperref}
\usepackage[dvipsnames]{xcolor}
\usepackage{pst-node}
\usepackage{tikz-cd} 
\usepackage{cancel}
\usepackage[colorinlistoftodos]{todonotes}

\allowdisplaybreaks

\theoremstyle{definition}
\newtheorem{theorem}{{Theorem}}[subsection]
\newtheorem{lemma}[theorem]{Lemma}
\newtheorem{proposition}[theorem]{Proposition}
\newtheorem{corollary}[theorem]{Corollary}
\newtheorem{notation}[theorem]{Notation}
\newtheorem{remark}[theorem]{Remark}
\newtheorem{definition}[theorem]{Definition}

\numberwithin{equation}{subsection}
\numberwithin{theorem}{subsection}
\newcount\colveccount
\newcommand*\colvec[1]{
        \global\colveccount#1
        \begin{pmatrix}
        \colvecnext
}
\def\colvecnext#1{
        #1
        \global\advance\colveccount-1
        \ifnum\colveccount>0
                \\
                \expandafter\colvecnext
        \else
                \end{pmatrix}
        \fi
}

\newcommand{\sA}{\mathsf{A}}
\newcommand{\sH}{\mathsf{H}}
\newcommand{\sP}{\mathsf{P}}
\newcommand{\sK}{\mathsf{K}}
\newcommand{\sZ}{\mathsf{Z}}
\newcommand{\sX}{\mathsf{X}}
\newcommand{\sY}{\mathsf{Y}}

\newcommand{\sT}{\mathsf{T}}

\newcommand{\sG}{\mathsf{G}}
\newcommand{\sV}{\mathsf{V}}
\newcommand{\sM}{\mathsf{M}}
\newcommand{\sR}{\mathsf{R}}
\newcommand{\sN}{\mathsf{N}}
\newcommand{\sS}{\mathsf{S}}
\newcommand{\sB}{\mathsf{B}}
\newcommand{\sD}{\mathsf{D}}
\newcommand{\sGL}{\mathsf{GL}}
\newcommand{\sSL}{\mathsf{SL}}
\newcommand{\sPSL}{\mathsf{PSL}}
\newcommand{\sSO}{\mathsf{SO}}

\newcommand{\cA}{\mathcal{A}}

\newcommand{\cX}{\mathcal{X}}
\newcommand{\cP}{\mathcal{P}}
\newcommand{\cH}{\mathcal{H}}

\newcommand{\cR}{\mathcal{R}}
\newcommand{\cS}{\mathcal{S}}

\newcommand{\cM}{\mathcal{M}}
\newcommand{\cE}{\mathcal{E}}

\newcommand{\tA}{\mathtt{A}}
\newcommand{\tAd}{\mathtt{Ad}}
\newcommand{\tad}{\mathtt{ad}}

\newcommand{\tC}{\mathtt{C}}
\newcommand{\tB}{\mathtt{B}}
\newcommand{\tF}{\mathtt{F}}
\newcommand{\tL}{\mathtt{L}}

\newcommand{\tT}{\mathtt{T}}
\newcommand{\tI}{\mathtt{I}}
\newcommand{\tR}{\mathtt{R}}
\newcommand{\tdR}{\dot{\mathtt{R}}}
\newcommand{\ttr}{\mathtt{tr}}
\newcommand{\texp}{\mathtt{exp}}
\newcommand{\tJd}{\mathtt{Jd}}
\newcommand{\tM}{\mathtt{M}}
\newcommand{\tP}{\mathtt{P}}

\newcommand{\tth}{\mathtt{h}}
\newcommand{\te}{\mathtt{e}}
\newcommand{\tu}{\mathtt{u}}
\newcommand{\N}{\mathbb{N}}
\newcommand{\R}{\mathbb{R}}
\newcommand{\C}{\mathbb{C}}

\newcommand{\sHom}{\mathsf{Hom}}

\newcommand{\flow}{\mathsf{U}\Gamma}
\newcommand{\cflow}{\widetilde{\mathsf{U}_0\Gamma}}
\newcommand{\bdry}{\partial_\infty\Gamma}
\newcommand{\fg}{\mathfrak{g}}
\newcommand{\fa}{\mathfrak{a}}

\newcommand{\fk}{\mathfrak{k}}

\newcommand{\fp}{\mathfrak{p}}
\newcommand{\fn}{\mathfrak{n}}
\newcommand{\fm}{\mathfrak{m}}

\newcommand{\GV}{\sG\ltimes\sV}
\newcommand\irr{\operatorname{irr}}
\newcommand\mdeg{\operatorname{mdeg}}

\title[Finite step Rigidity]{Finite step Rigidity of Hitchin representations and special Margulis-Smilga spacetimes}
\author{Sourav Ghosh}
\address{Ashoka University, Rajiv Gandhi Education City, Sonipat, Haryana 131029, India}
\email{sourav.ghosh@ashoka.edu.in, sourav.ghosh.bagui@gmail.com}
\date{\today}

\thanks{The author acknowledge support from the following grants: OPEN/16/11405402 and annual research grant of Ashoka University}
\begin{document}

\begin{abstract}
In this article we use the ``escape from subvarieties lemma" introduced by Eskin--Mozes--Oh to prove finite step rigidity results for the Jordan-Lyapunov projection spectra of Hitchin representations and the Margulis-Smilga invariant spectra of some special Margulis-Smilga spacetimes. In the process, we also prove a similar finite step rigidity result for the Cartan spectra of representations of a finitely generated group inside a connected semisimple real algebraic Lie group with  trivial center.
\end{abstract}

\maketitle
\tableofcontents

\newpage

\section*{Introduction}

Since antiquity Mathematicians are interested in trying to determine geometric objects uniquely using finite data. In this article, we try to do precisely that for geometric objects known as Hichin representations and Margulis-Smilga spacetimes. We show that Hitcin representations of a surface group are determined by the Jordan-Lyapunov projections of the images of a finite number of surface group elements and Margulis-Smilga spacetimes are determined by the knowledge of a finite number of (normed) Margulis-Smilga invariants. In fact, these results can also be used to give nice embeddings of the moduli space of Hitchin representations and the moduli space of Margulis-Smilga spacetimes. We use the ``escape from subvarieties lemma" of Eskin--Mozes--Oh (for more details please see Proposition 3.2 of \cite{EMO} and Lemma 4.2 of \cite{Br}) as the main ingredient to obtain these finite step rigidity results. In the process, we also prove a more general result for representations of a finitely generated group inside a connected semisimple real algebraic Lie group with  trivial center. We show that such representations are uniquely determined by the knowledge of finitely many Cartan projections. Unfortunately, unlike Jordan-Lyapunov projections, Cartan projections are a not a reliable marker of a representation as they are not conjugation invariant. The main novelty of this article is the usage of finite data to distinguish representations. Instead, if one allows checking of infinite data to determine a representation uniquely, then there exist a plethora of important results. 

On one hand, in \cite{DK} Dal'bo--Kim proved that Zariski dense representations of a discrete group inside a semisimple algebraic Lie group without compact factors is determined upto conjugacy if and only if both have the same Jordan-Lyapunov projection spectra. Later, in \cite{BCLS} Bridgeman--Canary--Labourie--Sambarino proved that two projective Anosov representations and in particular two Hitchin representations are conjugate to each other if and only if they have the same spectral radii spectra over all elements. Similar results involving all elements were also obtained by Kim \cite{Kim2} and Cooper--Delp \cite{CDel} in the context of Hilbert length spectra of projective manifolds. Moreover, in a recent work \cite{BCL} Bridgeman--Canary--Labourie show that two Hitchin representations are conjugate to each other if and only if they have the same marked trace spectra over simple non-separating elements. Similar results involving infinite data but only over simple non-separating elements were also obtained by Brideman--Canary \cite{BC} in the context of representations of surface groups inside $\sPSL_2(\C)$ and by Duchin--Leininger--Rafi \cite{DLR} in the context of flat surfaces.

On the other hand, in \cite{CD} Charette--Drumm proved that two Margulis spacetimes coming from proper affine actions on $\R^3$ are conjugate to each other if and only if they have the same Margulis invariant spectra. Later, this result was generalized by Kim \cite{Kim} to include proper affine actions coming from examples of Abels--Margulis--Soifer \cite{AMS}. Recently, Ghosh proved an infinitesimal version of Kim's result in \cite{Ghosh4} to define and study a pressure form on the moduli space and generalized Kim's result in \cite{Ghosh5} to include proper affine actions coming from examples of Smilga \cite{Smilga,Smilga3,Smilga4} where either the linear parts are conjugate to each other or the affine Lie group is obtained from a self-contragredient representation of its semisimple linear part.

In fact, all of these above results can also be seen as successor to the much celebrated result of Otal \cite{Otal} trying to answer if one could ``hear the shape of a drum". Otal proved that two negatively curved closed surfaces have the same marked length spectrum if and only if they are isometric (see also \cite{Croke}).

\subsection{Hitchin representations}

In \cite{Hit}, Hitchin used Higgs bundles techniques to study connected components of the space of reducible representations of a surface group inside $\sPSL(n,\R)$. Extending the results obtained in \cite{Goldman} by Goldman for $n=2$, he shows that for $n>2$ the space of reducible representations has three components when $n$ is odd and six components when $n$ is even. Among these components one component (respectively two components) when $n$ is odd (respectively $n$ is even) are special and topologically trivial. These components behave like Teichm\"uller spaces and are called Hitchin components. Elements of Hitchin components are called Hitchin representations. Later, in \cite{Labourie} Labourie extended results by Choi--Goldman \cite{CG1} and gave a nice geometric characterization of Hitchin representations in terms of Anosov representations. The definition of an Anosov representation introduced by Labourie is dynamical in nature (for more details please see Definition \ref{def.anosov}). Later work done by Kapovich--Leeb--Porti \cite{KLPmain}, Bochi--Potrie--Sambarino \cite{BPS} and Kassel--Potrie \cite{KaPo} imply that Anosov representations with respect to a maximal parabolic subgroup $\sP_i$ are precisely those representations whose elements have an uniform gap between the $i$-th and $(i+1)$-th coordinate of their Jordan-Lyapunov projections. We note that the Jordan-Lyapunov projection of an element is a precise and abstract way of listing the modulus of its eigenvalues in non-increasing order. Hence, Anosov representations with respect to minimal parabolic subgroups are those representations whose elements have an uniform gap between the $i$-th and $(i+1)$-th coordinate of their Jordan-Lyapunov projections for all $i$. In fact, combining works of Labourie \cite{Labourie} and Guichard \cite{Gui1} one deduces that Hitchin representations are precisely those representations which are Anosov with respect to minimal parabolic subgroups and which admit hyperconvex limit maps. 

In this article, we prove that
\begin{theorem}
	Let $\sG$ be a connected real split semisimple algebraic Lie group with trivial center and without compact factors. Let $\Gamma$ be a finitely generated word hyperbolic group with a finite symmetric generating set $\sS \subset \Gamma$, let $\sB_\sS(n)$ be the set of all elements of $\Gamma$ whose word length with respect to $\sS$ is less than some integer $n$ and let $\lambda$ be the Jordan-Lyapunov projection (for more details please see Definition \ref{def.jordan}). Then, there exists $m\in\N$ depending only on $\sG$ such that any two representations $\rho_1,\rho_2: \Gamma \to \sG$ which satisfy the following:
	\begin{enumerate}
		\item $\rho_i(\Gamma)$ is Zariski dense in $\sG$,
		\item $\rho_i(\Gamma)$ is Anosov with respect to minimal parabolic subgroups,
		\item $\lambda(\rho_1(\gamma))=\lambda(\rho_2(\gamma))$ for all $\gamma \in \sB_\sS(m)$,
	\end{enumerate}
	are equivalent; in other words, there exists an automorphism $\sigma:\sG \to \sG$ such that we have $\sigma \circ \rho_1=\rho_2$.
\end{theorem}

Moreover, using the above result and the fact that Hitchin representations are special cases of Anosov representations with respect to minimal parabolic subgroups, we prove the following:

\begin{theorem}
	Let $\Gamma$ be the fundamental group of a compact surface without boundary, $\sS\subset\Gamma$ be a finite symmetric generating set, $\sB_\sS(n)$ be the set of all elements of $\Gamma$ whose word length with respect to $\sS$ is less than some integer $n$, let $\lambda$ be the Jordan-Lyapunov projection and let $\cH_n(\Gamma)$ be the space of Hitchin representations of $\Gamma$ into $\sPSL(n,\R)$. Then, there exists $m_n\in\N$ such that any two representations $\rho_1,\rho_2\in\cH_n(\Gamma)$ which satisfy the following: $\lambda(\rho_1(\gamma))=\lambda(\rho_2(\gamma))$ for all $\gamma \in \sB_\sS(m_n)$, are equivalent; in other words, there exists an automorphism $\sigma:\sG \to \sG$ such that we have $\sigma \circ \rho_1=\rho_2$.
\end{theorem}

\subsection{Margulis-Smilga spacetimes}

While trying to answer a question asked by Milnor \cite{Mil}, in \cite{Margulis1,Margulis2} Margulis showed that non-abelian free groups admit proper discontinuous actions on $\R^3$ via affine transformations. The existence of such actions of non-abelian free groups came as a surprise to the community due to the Auslander's conjecture \cite{Aus2}. Auslander's conjecture states that the only subgroups of $\sGL_n(\R)\ltimes\R^n$ which act properly discontinuously, freely and cocompactly on $\R^n$ are virtually polycyclic. Auslander came up with this conjecture, based on his work with Markus \cite{AM}, while trying to generalize Bieberbach's results \cite{B1,B2} on classification of subgroups of $\mathsf{O}_n(\R)\ltimes\R^n$ which act properly discontinuously, freely and cocompactly on $\R^n$. Auslander's conjecture has been shown to hold true for $n\leq6$ by the work of Fried--Goldman \cite{FG} and Abels--Margulis--Soifer \cite{AMS2}.

Margulis' results were further extended by Abels--Margulis--Soifer \cite{AMS}. They showed that there exist non-abelian free subgroups of $\mathsf{SO}(n,n+1)\ltimes\R^{2n+1}$ which act properly dicontinuously and freely on $\R^{2n+1}$ if and only if $n$ is odd. In all these results the notion of a real valued invariant, subsequently known as Margulis invariant, played a central role in determining properness. In \cite{AMS3} Abels--Margulis--Soifer introduced a new vector valued invariant to replace the Margulis invariant. A similar invariant was later extensively used by Smilga \cite{Smilga,Smilga3,Smilga4} to prove existence of proper actions of non-abelian free subgroups of more general affine groups. In this article, the examples of proper affine actions of non-abelian free groups coming from the work of Smilga, generalizing previous work by Margulis and Abels--Margulis--Soifer, has been called Margulis-Smilga spacetimes and the vector valued versions of the Margulis invariants needed to gauge properness has been called Margulis-Smilga invariants. 

In this article we prove the following:

\begin{theorem}\label{thm.msinv}
	Let $\sG$ be a connected real split semisimple algebraic Lie group with trivial center and without compact factors. Let $\Gamma$ be a finitely generated group with a finite symmetric generating set $\sS \subset \Gamma$, let $\sB_\sS(n)$ be the set of all elements of $\Gamma$ whose word length with respect to $\sS$ is less than some integer $n$ and let $\tM$ be the Margulis-Smilga invariant (for more dertails please see Definition \ref{def.margi}). Then, there exists $m\in\N$ depending only on $\sG\ltimes_\tR\sV$ such that any two Margulis-Smilga spacetimes $\rho_1,\rho_2:\Gamma\to\sG\ltimes_\tR\sV$ whose linear parts are conjugate to $\varrho$ and which satisfy the following:
	\begin{enumerate}
		\item $\varrho(\Gamma)$ is Zariski dense in $\sG$,
		\item $\tM(\rho_1(\gamma))=\tM(\rho_2(\gamma))$ for all $\gamma \in \sB_\sS(m)$,
	\end{enumerate}
	are equivalent; in other words, there exists $X\in\sV$ such that we have $(g,X)^{-1}\rho_1(g,X)=\rho_2$.
\end{theorem}

Moreover, using similar methods we also prove the following:

\begin{theorem}\label{thm.normmsinv}
	Let $\sG$ be a connected real split semisimple algebraic Lie group with trivial center and without compact factors, let $\tR$ preserves a norm $\|\cdot\|_\tR$ on $\sV$. Let $\Gamma$ be a finitely generated group with a finite symmetric generating set $\sS \subset \Gamma$, , let $\sB_\sS(n)$ be the set of all elements of $\Gamma$ whose word length with respect to $\sS$ is less than some integer $n$ and let $\tM$ be the Margulis-Smilga invariant. Then, there exists $m\in\N$ depending only on $\sG\ltimes_\tR\sV$ such that any two Margulis-Smilga spacetimes $\rho_1,\rho_2:\Gamma\to\sG\ltimes_\tR\sV$ whose linear parts are Zariski dense in $\sG$ and which satisfy the following: $\|\tM(\rho_1(\gamma))\|_\tR=\|\tM(\rho_2(\gamma))\|_\tR$ for all $\gamma \in \sB_\sS(m)$, are equivalent, i.e. there exists $(\tA,X)\in\sGL(\sV)\ltimes\sV$ such that we have $(\tA,X)^{-1}\rho_1(\tA,X)=\rho_2$.
\end{theorem}

We observe that for the special case of self-contragradient representations Theorem \ref{thm.msinv} follows directly from Theorem \ref{thm.normmsinv}.

\subsection{Key Ingredient}

The main ingredient that is repeatedly used in this article to obtain the various finite step rigidity results mentioned above, is a variation of the following escape from subvarieties lemma.

\begin{lemma}[Eskin--Mozes--Oh, see Proposition 3.2 \cite{EMO} and Lemma 4.2 \cite{Br}]\label{lemma.escape}
Let $\Gamma\subset \sGL_n(\C)$ be any finitely generated subgroup, let $\sS$ be a finite set of generators of $\Gamma$, let $\sB_\sS(n)$ be the set of all elements of $\Gamma$ whose word length with respect to $\sS$ is less than some integer $n$  and let $\sH$ be the Zariski closure of $\Gamma$, which is also assumed to be Zariski-connected. Then for any proper subvariety $\sX$ of $\sH$, there exists a positive integer $M$ depending only on $\sX$ such that for any finite generating set $\sS$ of $\Gamma$, we have $\sB_\sS(M)\not\subset \sX$.
\end{lemma}

In this article, we closely follow the original proof of the escape from subvarieties lemma and reprove it to obtain the explicit bounds in the aforementioned finite step rigidity results. In fact, exact formulas for the degrees of various representations of Lie groups appearing in this text can be found in \cite{Kazar} and Proposition 4.7.18 of \cite{DeKe}. Moreover, for the particular case of $\sSO(n,\C)$, an explicit formula can also be found in \cite{BBBKR}.

\section*{Acknowledgements}
I would like to warmly thank \c Ca\u gri Sert for informing me about the escape from subvarieties lemma and explaining it to me. This article grew out of discussions with him. I would also like to thank CIRM, Luminy for hosting us.

\section{Escape from subvarieties}
In this section we will reprove the escape from subvarieties lemma due to Eskin--Mozes--Oh \cite{EMO} in order to obtain explicit bounds. We will closely follow the arguments given in \cite{EMO} but will keep track of the bounds to obtain the needed constants.

\subsection{Inequalities}\label{subsec.ineq}

We now define certain terminologies and use them to provide an explicit description of the constant $M$ appearing in Lemma \ref{lemma.escape}. Let $\sX$ be an algebraic variety. If $\sX=\cup_{i=1}^l \sX_i$ and for $1\leq i \leq l$, $\sX_i$'s are the irreducible components of $\sX$. We define $d(\sX) := \max_i\dim(\sX_i)$, $\irr(\sX)$ to be the number of irreducible components of $\sX$ i.e $\irr(\sX)=l$, $\irr_{md}(\sX)$ to be the number of irreducible components of $\sX$ of the maximal dimension $d(\sX)$ and $\mdeg(\sX)$ to be the maximal degree of an irreducible component of $\sX$. In the following lemma we state and prove a modified version of Lemma 3.5 of \cite{EMO} and use its help to obtain an explicit description of $M$:
\begin{lemma}\label{lem:ineq}
Let $\Gamma\subset \sGL_n(\C)$ be a subgroup with a finite generating set $\sS$. Moreover, let $\sH$ be the Zariski closure of $\Gamma$, which is also assumed to be Zariski-connected and let $\sY$ be a proper subvariety of $\sH$. Then there exists an integer $m\leq\irr_{md}(\sY)$ and a sequence of $m$ elements $s_0, s_1,\dots, s_{m-1}$ of $\sS$ so that if we define the following sequence of varieties: $\sV_0\sY=\sY$, $\sV_{i+1}\sY=\sV_i\sY\cap s_i \sV_i\sY$, for $0\leq i\leq m-1$, then $\sV_m\sY$ satisfies $d(\sV_m\sY)<d(\sY)$. Moreover, the following inequalities hold:
\begin{align*}
    \irr(\sV_m\sY)&\leq\irr(\sY)^{2^{\irr(\sY)}}\mdeg(\sY)^{\irr(\sY)2^{\irr(\sY)}},\\
    \mdeg(\sV_m\sY)&\leq\mdeg(\sY)^{2^{\irr(\sY)}}.
\end{align*}
\end{lemma}
\begin{proof}
We start by proving the first part of the Lemma using induction on $\irr_{md}(\sY)$. Let $\sV_0\sY=\sY$. If $\irr_{md}(\sY) = 1$ then by Lemma 3.3 \cite{EMO} there exists $s_0\in \sS$ such that $d(\sY\cap s_0\sY)<d(\sY)$. Hence for $\sV_1\sY:=\sY\cap s_0 \sY$ we have $d(\sV_1\sY)<d(\sY)$ with $1\leq\irr_{md}(\sY)$ and we are done. By induction hypothesis on $\irr_{md}(\sY)$ we assume that for all $\sY$ with $\irr_{md}(\sY)\leq n$ there exists an integer $m\leq \irr_{md}(\sY)$ and a sequence of $m$ elements $s_0, s_1,\dots, s_{m-1}$ of $\sS$ so that if we define the following sequence of varieties: $\sV_0\sY=\sY$, $\sV_{i+1}\sY=\sV_i\sY\cap s_i \sV_iY$, for $0\leq i\leq m-1$, then $\sV_m\sY$ satisfies $d(\sV_m\sY)<d(\sY)$. Now let $\sY$ be such that $\irr_{md}(\sY)=n+1$. Then by Lemma 3.4 \cite{EMO} there exists $s\in \sS$ such that either $d(\sY\cap s\sY)<d(\sY)$ or $\irr_{md}(\sY\cap s\sY)<\irr_{md}(\sY)$. If $d(\sY\cap s\sY)<d(\sY)$ then for $\sV_1\sY:=\sY\cap s\sY$ we have $1\leq n+1$ and we are done. Otherwise, $d(\sY\cap s\sY)=d(\sY)$ and $\irr_{md}(\sY\cap s\sY)<\irr_{md}(\sY)$. Hence for $\sZ=\sY\cap s\sY$ we have $\irr_{md}(\sZ)\leq n$ and by induction hypothesis we obtain that there there exists an integer $m\leq \irr_{md}(\sZ)$ and a sequence of $m$ elements $s_0, s_1,\dots, s_{m-1}$ of $\sS$ so that if we define the following sequence of varieties: $\sV_0\sZ=\sZ$, $\sV_{i+1}\sZ=\sV_i\sZ\cap s_i \sV_iZ$, for $0\leq i\leq m-1$, then $\sV_m\sZ$ satisfies $d(\sV_m\sZ)<d(\sZ)$. Now, if we define $\sV_0\sY=\sY$ and $\sV_{i+1}\sY:=\sV_i\sZ$ then 
\[\sV_{i+1}\sY=\sV_i\sZ=\sV_{i-1}\sZ\cap s_{i-1} \sV_{i-1}\sZ=\sV_i\sY\cap s_{i-1} \sV_i\sY,\]
$d(\sV_{m+1}\sY)<d(\sZ)=d(\sY\cap s\sY)=d(\sY)$ and as $\irr_{md}(\sZ)<\irr_{md}(\sY)$ we obtain $m+1\leq\irr_{md}(\sZ)+1\leq\irr_{md}(\sY)$. Therefore, we obtain that for any proper subvariety $\sY$ of $\sH$ there exists an integer $m\leq\irr_{md}(\sY)$ and a sequence of $m$ elements $s_0, s_1,\dots, s_{m-1}$ of $\sS$ so that if we define the following sequence of varieties: $\sV_0\sY=\sY$, $\sV_{i+1}\sY=\sV_i\sY\cap s_i \sV_iY$, for $0\leq i\leq m-1$, then $\sV_m\sY$ satisfies $d(\sV_m\sY)<d(\sY)$. This concludes the first part of this lemma.

Now we prove the inequalities. Let $\sZ_{i,j}$ be the irreducible components of $\sV_i\sY$ for $1\leq j\leq n_i$. We observe that 
\begin{align*}
    \sV_{i+1}\sY=\left(\bigcup_{j=1}^{n_i}\sZ_{i,j}\right)\bigcap\left(\bigcup_{k=1}^{n_i}s_i\sZ_{i,k}\right)=\bigcup_{j=1}^{n_i}\bigcup_{k=1}^{n_i}(\sZ_{i,j}\cap s_i\sZ_{i,k}).
\end{align*}
Hence $n_{i+1}=\irr(\sV_{i+1}\sY)=\sum_{j=1}^{n_i}\sum_{k=1}^{n_i}\irr(\sZ_{i,j}\cap s_i\sZ_{i,k})$. We use the Generalized Bezout's Theorem (page 519 \cite{Sch}, Theorem 2.2 at page 251 \cite{Dl}, Proposition 2.3 at page 10 \cite{Fulton}) and obtain that for fixed $i,j$ and $k$:
\begin{align*}
   &\max\{\mdeg(\sZ_{i,j}\cap s_i\sZ_{i,k}), \irr(\sZ_{i,j}\cap s_i\sZ_{i,k})\}\\
   &\leq \deg{\sZ_{i,j}}\deg{s_i\sZ_{i,j}}=\deg{\sZ_{i,j}}^2\leq \mdeg{(\sV_i\sY)}^2.
\end{align*}
Therefore, it follows that $n_{i+1}\leq n_i^2\mdeg{(\sV_iY)}^2$ and 
\begin{align*}
    \mdeg{(\sV_{i+1}\sY)}=\max_{j,k}\{\mdeg{(\sZ_{i,j}\cap s_i\sZ_{i,k})}\}\leq \mdeg{(\sV_i\sY)}^2.
\end{align*}
We observe that $\irr(\sV_0Y)=\irr(\sY)\leq\irr(\sY)=\irr(\sY)^{2^0}\mdeg(\sY)^{0.2^0}$ and $\mdeg(\sV_0\sY)=\mdeg(\sY)\leq\mdeg(\sY)=\mdeg(\sY)^{2^0}$. Moreover, by induction hypothesis we assume that 
\begin{align*}
    \irr(\sV_i\sY)&\leq\irr(\sY)^{2^i}\mdeg(\sY)^{i.2^i},\\
    \mdeg(\sV_i\sY)&\leq\mdeg(\sY)^{2^i}.
\end{align*}
Then we obtain the following
\begin{align*}
    \irr(\sV_{i+1}\sY)=n_{i+1}&\leq n_i^2\mdeg{(\sV_i\sY)}^2\\
    &\leq\irr(\sY)^{2^i.2}\mdeg(\sY)^{i.2^i.2}\mdeg(\sY)^{2^i.2}\\
    &=\irr(\sY)^{2^{i+1}}\mdeg(\sY)^{(i+1).2^{i+1}},\\
    \mdeg(\sV_{i+1}\sY)&\leq\mdeg{(\sV_i\sY)}^2\leq\mdeg(\sY)^{2^i.2}=\mdeg(\sY)^{2^{i+1}}.
\end{align*}
Therefore, by induction we obtain that
\begin{align*}
    \irr(\sV_m(\sY))\leq \irr(\sY)^{2^m}\mdeg(\sY)^{m.2^m} ,\\
    \mdeg(\sV_m(\sY))\leq \mdeg(\sY)^{2^m}.
\end{align*}
Finally, we recall that $m\leq\irr_{md}(\sY)$ and conclude the inequalities by observing that $\irr_{md}(\sY)\leq\irr(\sY)$.
\end{proof}

Let us consider the following collection of functions $f_{x,y}:\N\to\N\times\N$ for $x, y\in\N$ such that $f_{x,y}(n):=(x_n,y_n)$ and 
\begin{align*}
    (x_0,y_0)&:= (x,y),\\
    (x_{n+1},y_{n+1})&:=(x_n^{2^{x_n}}y_n^{x_n2^{x_n}},y_n^{2^{x_n}}).
\end{align*}
We denote the projection onto the first coordinate from $\N\times\N$ to $\N$ by $\pi_1$ i.e. $\pi_1(a,b):=a$. Moreover, we define $M:\N\times\N\times\N\to\N$ as follows:
\begin{align*}
    M(x,y,n):=\sum_{i=0}^{n}\pi_1\circ f_{x,y}(i).
\end{align*}
\begin{remark}
Let $\sX$ be a proper subvariety of $\sGL_n(\C)$. Then by abuse of notation we denote: $M_\sX:=M(\irr(\sX),\mdeg(\sX),d(\sX))$.
\end{remark}
\begin{remark}
Let $a,b,A,B\in\N$. Then we say $(a,b)\leq(A,B)$ if and only if both $a\leq A$ and $b\leq B$.
\end{remark}
\begin{lemma}\label{lem:f}
Let $f_{x,y}$ be as above for $x, y\in\N$ with $f_{x,y}(n)=(x_n,y_n)$ and let $a,b,A,B\in\N$ be such that $(a,b)\leq (A^{2^A}B^{A2^A}, B^{2^A})$. Then for all non-negative integers $n$ we have $f_{a,b}(n)\leq f_{A,B}(n+1)$.
\end{lemma}
\begin{proof}
We will prove this by induction on $n$. As $(a,b)\leq (A^{2^A}B^{A2^A}, B^{2^A})$ we obtain that $f_{a,b}(0)\leq f_{A,B}(1)$. Now by induction hypothesis we assume that for all $0\leq i\leq n-1$ we have
$f_{a,b}(i)\leq f_{A,B}(i+1)$. Then
\begin{align*}
    f_{a,b}(n)&=(a_{n-1}^{2^{a_{n-1}}}b_{n-1}^{a_{n-1}2^{a_{n-1}}},b_{n-1}^{2^{a_{n-1}}})\\
    &\leq (A_n^{2^{A_n}}B_n^{A_n2^{A_n}},B_n^{2^{A_n}})=f_{A,B}(n+1).
\end{align*}
Hence we conclude our result using induction.
\end{proof}

\begin{lemma}\label{lem:M}
Let $\Gamma\subset \sGL_n(\C)$ be a subgroup generated by a finite set $\sS$, let $\sH$ be the Zariski closure of $\Gamma$, which is also assumed to be Zariski-connected and let $\sX$ be a proper subvariety of $\sH$. Moreover, let $m\leq \irr(\sX)$ and $\{s_i\}_{i=0}^{m-1}\subset \sS$ be such that $d(\sV_m\sX)<d(\sX)$, where $\sV_m\sX$ is defined inductively as follows: $\sV_0\sX=\sX$ and $\sV_{i+1}\sX=\sV_i\sX\cap s_i \sV_i\sX$, for $0\leq i\leq m-1$. Then
\[M_{\sV_m\sX}+\irr(\sX)\leq M_\sX.\]
\end{lemma}
\begin{proof}
As $d(\sV_m\sX)<d(\sX)$ and they are integers, we get $d(\sV_m\sX)+1\leq d(\sX)$. Now we use Lemmas \ref{lem:ineq}, \ref{lem:f} to obtain the following: 
\begin{align*}
    M_{\sV_m\sX}&=\sum_{i=0}^{d(\sV_m\sX)}\pi_1\circ f_{\irr(\sV_m\sX),\mdeg(\sV_m\sX)}(i)\leq \sum_{i=0}^{d(\sV_m\sX)}\pi_1\circ f_{\irr(\sX),\mdeg(\sX)}(i+1)\\
    &=\sum_{i=1}^{d(\sV_m\sX)+1}\pi_1\circ f_{\irr(\sX),\mdeg(\sX)}(i)\leq\sum_{i=1}^{d(\sX)}\pi_1\circ f_{\irr(\sX),\mdeg(\sX)}(i)\\
    &=M_\sX-\irr(\sX).
\end{align*}
This concludes our result.
\end{proof}

\subsection{Explicit bounds}

In this subsection, we use the inequalities from Subsection \ref{subsec.ineq} and obtain the explicit bounds that will be needed in proving the main results of this article.

\begin{proposition}\label{prop.main}
Let $\Gamma\subset \sGL_n(\C)$ be any finitely generated subgroup, let $\sH$ be the Zariski closure of $\Gamma$, which is also assumed to be Zariski-connected and let $\sX$ be a proper subvariety of $\sH$. Then for any finite generating set $\sS$ of $\Gamma$, we have $\sB_\sS(M_\sX)\not\subset \sX$.
\end{proposition}
\begin{proof}
Let $\sS$ be a finite generating set of $\Gamma$. Let $\sX_0\supset \sX_1\supset \dots \supset \sX_k$ be a decreasing tower of subvarieties such that $\sX_0=\sX$, $\sX_k$ is the empty variety and $\sX_{i+1}$ is a subvariety obtained from $\sX_i$ by using Lemma \ref{lem:ineq} i.e. $\sX_{i+1}=\sV_{m_i}(\sX_i)$, $m_i\leq\irr(\sX_i)$ and $d(\sX_{i+1})<d(\sX_i)$. Clearly, $k$ is bounded above by $d(\sX)$. 

We will prove this result by induction on $d(\sX)$. If $d(\sX)=1$ then $\sX_1=\sV_{m_1}\sX$ is an empty variety and we obtain
\begin{align*}
    m_1\leq\irr(\sX)\leq \irr(\sX) + M_{\sV_{m_1}\sX}\leq M_\sX.
\end{align*}
Also, for some words $\{w_i\}_{i=1}^t\subset\Gamma$ with word length not bigger than $m_1$ with respect to $\sS$, we have $\emptyset=\sV_{m_1}\sX=\cap_{i=1}^tw_i\sX$. Now if $\sB_\sS(M_\sX)\subset \sX$ then we have $\{e,w_i^{-1}\}_{i=1}^t\subset \sX$ and hence we obtain the following contradiction: $\{e\}\subset\cap_{i=1}^tw_i\sX=\emptyset$. Therefore, $\sB_\sS(M_\sX)\not\subset \sX$. 

By induction hypothesis let us assume that for all $\sX\subset \sH$ with $d(\sX)\leq n$ we have $\sB_\sS(M_\sX)\not\subset \sX$. Now let $\sX\subset \sH$ with $d(\sX)= n+1$. We use Lemma \ref{lem:ineq} and obtain that there exists an integer $m\leq\irr_{md}(\sX)$ and a sequence of $m$ elements $s_0, s_1,\dots, s_{m-1}$ of $\sS$ so that if we define the following sequence of varieties: $\sV_0\sX=\sX$, $\sV_{i+1}\sX=\sV_i\sX\cap s_i \sV_iX$, for $0\leq i\leq m-1$, then $\sV_m\sX$ satisfies $d(\sV_m\sX)<d(\sX)$. Hence, $d(\sV_m\sX)\leq d(\sX)-1= n$ and by the induction hypothesis we obtain that $\sB_\sS(M_{\sV_m\sX})\not\subset \sV_m\sX$. We want to show that $\sB_\sS(M_\sX)\not\subset \sX$. We will prove it via contradiction. If possible let us assume that $\sB_\sS(M_{\sV_m\sX})\not\subset \sV_m\sX$ and $\sB_\sS(M_\sX)\subset \sX$. As $\sB_\sS(M_{\sV_m\sX})\not\subset \sV_m\sX$, there exists a word $w\in\Gamma$ of word length less than $M_{\sV_m\sX}$ with respect to $\sS$, such that $w\not\in \sV_m\sX$. Moreover, there exist words $\{w_i\}_{i=1}^t\subset\Gamma$ with word length not bigger than $m$ with respect to $\sS$, such that $\sV_m\sX=\cap_{i=1}^tw_i\sX$. We note that the word length of the words $\{w_i^{-1}w\}_{i=1}^t$ is not bigger than $m+M_{\sV_m\sX}$ and hence by Lemmas \ref{lem:ineq} and \ref{lem:M} is bounded by $M_\sX$. Therefore, $\{w_i^{-1}w\}_{i=1}^t\subset \sB_\sS(M_\sX)\subset \sX$ and we obtain that
$\{w\}\subset\cap_{i=1}^tw_i\sX=\sV_m(\sX)$, a contradiction. We conclude that for $\sX\subset \sH$ with $d(\sX)= n+1$ we have $\sB_\sS(M_\sX)\not\subset \sX$. 

Finally, using induction on dimension of $\sX$ we conclude our result.
\end{proof}

\begin{proposition}\label{prop.linfinite}
Let $\sG$ be a connected semisimple real algebraic Lie group with trivial center and without compact factors. Let $\Gamma$ be a finitely generated group and let $\rho_1,\rho_2: \Gamma \to \sG$ be two representations such that $\rho_1(\Gamma)$, $\rho_2(\Gamma)$ are Zariski dense in $\sG$. Moreover, let $\sX$ be the Zariski closure of $(\rho_1,\rho_2)(\Gamma)$ inside $\sG\times\sG$. Then the following holds:
\begin{enumerate}
    \item There exist two normal subgroups $\sN_1,\sN_2$ of $\sG$ and a smooth isomorphism $\sigma:\sG/\sN_1\to\sG/\sN_2$ such that 
    \[\sX=\{(g,h)\mid h\sN_2=\sigma(g\sN_1)\}.\]
    \item Let $\cS$ be the set of all $\sY\subset\sG\times\sG$ for which there exist two normal subgroups $\sN_1,\sN_2$ of $\sG$ and a smooth isomorphism $\sigma:\sG/\sN_1\to\sG/\sN_2$ with
    \[\sY=\{(g,h)\mid h\sN_2=\sigma(g\sN_1)\}\]
    then $\max\{M_\sY\mid\sY\in\cS\}$ exists and is finite.
\end{enumerate}
\end{proposition}
\begin{proof}
We will prove this result in two parts.
 
$\diamond$ Let $\pi_i: \sG \times \sG \to \sG$ denote the projection onto the $i^{th}$-factor for $i=1,2$. We observe that for any $\gamma\in\Gamma$ the group $(\rho_1,\rho_2)(\Gamma)$ is normalized by $(\rho_1(\gamma),\rho_2(\gamma))$. Hence $\sX$ is normalized by $(\rho_1(\gamma),\rho_2(\gamma))$ for all $\gamma\in\Gamma$. It follows that for all $i\in\{1,2\}$, $\pi_i(\sX)$ is a normal subgroup of $\sG$. Also, as $\pi_i(\sX)\supset\rho_i(\Gamma)$ with $\rho_i(\Gamma)$ Zariski dense in $\sG$ we obtain that $\pi_1(\sX)=\sG=\pi_2(\sX)$. We denote the kernels of $\left.\pi_1\right|_\sX,\left.\pi_2\right|_\sX$ by $\sN_1,\sN_2$ respectively and observe that $\sN_1=\sX\cap(\{e\}\times\sG)$, $\sN_2=\sX\cap(\sG\times\{e\})$ and
\begin{align*}
    (e,h)\sN_1(e,h)^{-1}&=(g,h)\sN_1(g,h)^{-1}=\sN_1,\\
    (g,e)\sN_2(g,e)^{-1}&=(g,h)\sN_2(g,h)^{-1}=\sN_2,
\end{align*}
for all $(g,h)\in\sX$. It follows that $\sN_1$ is normal inside $\{e\}\times\sG$ and $\sN_2$ is normal inside $\sG\times\{e\}$. Hence by Goursat's Lemma \cite{Gour} we obtain that the image of $\sX$ inside $(\sG\times\sG)/(\sN_1\sN_2)$ is the graph of an isomorphism $\sigma: \sG/\sN_1\to\sG/\sN_2$. Moreover, as $\pi_1$ and $\pi_2$ are smooth maps and $\sX$ is an algebraic group, we deduce that $\sigma$ is smooth.

$\diamond$ As $\sN_1,\sN_2$ are normal subgroups of a connected semisimple real algebraic Lie group with trivial center and without compact factors, we obtain that there are finitely many such normal subgroups (see Theorem 4.3 of \cite{Milne}). Moreover, as any smooth isomorphism between two connected Lie groups $\sG_1$ and $\sG_2$ uniquely determine an invertible linear transformation between $\fg_1$ and $\fg_2$, we have $\mathsf{Isom}(\sG_1,\sG_2)\subset\sHom(\fg_1,\fg_2)$. We also know that invertible linear transformations do not change the number of irreducible components, the degrees and the dimensions of an algebraic variety. Hence $\{M_\sY\mid\sY\in\cS\}$ is a finite set and we conclude that $\max\{M_\sY\mid\sY\in\cS\}$ exists and is finite.
\end{proof}

\begin{proposition}\label{prop.affinite}
Let $\sG$ be a connected semisimple real algebraic Lie group with trivial center and without compact factors. Let $\sV$ be a finite dimensional vector space and $\tR:\sG\to\sGL(\sV)$ be an irreducible algebraic representation. Let $\Gamma$ be a finitely generated group and let $\rho_1,\rho_2: \Gamma \to \sG$ be two representations such that $\rho_1(\Gamma)$, $\rho_2(\Gamma)$ are Zariski dense in $\sG\ltimes_\tR\sV$. Moreover, let $\sX$ be the Zariski closure of $(\rho_1,\rho_2)(\Gamma)$ inside $(\sG\ltimes_\tR\sV)\times(\sG\ltimes_\tR\sV)$. Then the following holds:
\begin{enumerate}
    \item There exist two normal subgroups $\sN_1,\sN_2$ of $\sG\ltimes_\tR\sV$ and a smooth isomorphism $\sigma:(\sG\ltimes_\tR\sV)/\sN_1\to(\sG\ltimes_\tR\sV)/\sN_2$ such that
    \[\sX=\{(g,X,h,Y)\mid (h,Y)\sN_2=\sigma((g,X)\sN_1)\}.\]
    \item Let $\cS$ be the set of all $\sY\subset\sG\times\sG$ for which there exist two normal subgroups $\sN_1,\sN_2$ of $\sG\ltimes_\tR\sV$ and a smooth isomorphism $\sigma:(\sG\ltimes_\tR\sV)/\sN_1\to(\sG\ltimes_\tR\sV)/\sN_2$ with
    \[\sY=\{(g,X,h,Y)\mid (h,Y)\sN_2=\sigma((g,X)\sN_1)\}\]
    then $\max\{M_\sY\mid\sY\in\cS\}$ exists and is finite.
\end{enumerate}
\end{proposition}
\begin{proof}
We will prove this result in two parts.
 
$\diamond$ Let $\pi_i: (\sG\ltimes_\tR\sV) \times (\sG\ltimes_\tR\sV) \to (\sG\ltimes_\tR\sV)$ denote the projection onto the $i^{th}$-factor for $i=1,2$. We observe that for any $\gamma\in\Gamma$ the group $(\rho_1,\rho_2)(\Gamma)$ is normalized by $(\rho_1(\gamma),\rho_2(\gamma))$. Hence $\sX$ is normalized by $(\rho_1(\gamma),\rho_2(\gamma))$ for all $\gamma\in\Gamma$. It follows that for all $i\in\{1,2\}$, $\pi_i(\sX)$ is a normal subgroup of $\sG\ltimes_\tR\sV$. Also, as $\pi_i(\sX)\supset\rho_i(\Gamma)$ with $\rho_i(\Gamma)$ Zariski dense in $\sG\ltimes_\tR\sV$ we obtain that $\pi_1(\sX)=(\sG\ltimes_\tR\sV)=\pi_2(\sX)$. We denote the kernels of $\pi_1,\pi_2$ by $\sN_1,\sN_2$ respectively and observe that $\sN_1=\sX\cap(\{(e,0)\}\times(\sG\ltimes_\tR\sV))$, $\sN_2=\sX\cap((\sG\ltimes_\tR\sV)\times\{(e,0)\})$ and
\begin{align*}
    (e,0,h,Y)\sN_1(e,0,h,Y)^{-1}&=(g,X,h,Y)\sN_1(g,X,h,Y)^{-1}=\sN_1,\\
    (g,X,e,0)\sN_2(g,X,e,0)^{-1}&=(g,X,h,Y)\sN_2(g,X,h,Y)^{-1}=\sN_2,
\end{align*}
for all $(g,X,h,Y)\in\sX$. It follows that $\sN_1$ is normal inside $\{(e,0)\}\times(\sG\ltimes_\tR\sV)$ and $\sN_2$ is normal inside $(\sG\ltimes_\tR\sV)\times\{(e,0)\}$. Hence by Goursat's Lemma \cite{Gour} we obtain that the image of $\sX$ inside $((\sG\ltimes_\tR\sV)\times(\sG\ltimes_\tR\sV))/(\sN_1\sN_2)$ is the graph of an isomorphism $\sigma: (\sG\ltimes_\tR\sV)/\sN_1\to(\sG\ltimes_\tR\sV)/\sN_2$. Moreover, as $\pi_1$ and $\pi_2$ are smooth and $\sX$ is an algebraic group, we deduce that $\sigma$ is smooth.

$\diamond$ We use Proposition A.2 of \cite{Ghosh5} to obtain that $\sN_1$, $\sN_2$ are trivial or $\sN_1=\sG_1\ltimes_\tR\sV$ and $\sN_2=\sG_2\ltimes_\tR\sV$ for some normal subgroups $\sG_1,\sG_2$ of $\sG$. Hence there are finitely many such normal subgroups (see Theorem 4.3 of \cite{Milne}). Moreover, we note that any smooth isomorphism between two connected Lie groups is linear and any invertible linear transformation does not change the number of irreducible components, the degrees and the dimensions of an algebraic variety. Hence $\{M_\sY\mid\sY\in\cS\}$ is a finite set and we conclude that $\max\{M_\sY\mid\sY\in\cS\}$ exists and is finite.
\end{proof}

\section{Finite step rigidity}\label{sec.fsr}

In this section, we prove finite step rigidity results for the Cartan spectra, Jordan spectra of representations in semisimple Lie groups and finite step rigidity results for the Margulis-Smilga spectra of representations in affine Lie groups satisfying certain criteria.

\subsection{Cartan spectrum} \label{subsec.carpro}

In this subsection we will prove finite step rigidity of the Cartan projection spectrum of representations.

Let $\sG$ be a real semisimple linear Lie group of noncompact type and with trivial center and let $\fg$ be its Lie algebra. We denote the conjugation map on $\sG$ by $\tC_g$ and let $\tAd_g$ be the differential at identity of the conjugation map by $g$ i.e. for $e\in \sG$, the identity element, and for any $g,h\in \sG$ we have $\tC_g(h):=ghg^{-1}$ and $\tAd_g:=d_e\tC_g$. We note that the map $\tAd: \sG \to \sSL(\fg)$ is a homomorphism. Moreover, let $\tad$ be the differential of $\tAd$ at the identity element i.e. $\tad=d_e\tAd$ . We fix a Cartan involution $\theta:\fg\to\fg$ and consider the corresponding decomposition $\fg=\fk\oplus\fp$ where $\fp$ (respectively $\fk$) is the eigenspace of eigenvalue -1 (respectively 1). We denote the maximal abelian subspace of $\fp$ by $\fa$ and let $\fa^*$ be the space of linear forms on $\fa$. We define
\[\fg_\alpha:=\{X\in\fg\mid \tad_H(X)=\alpha(H)X\text{ for all } H\in\fa\}\]
for all $\alpha\in\fa^*$ and call $\alpha\in\fa^*$ a \textit{restricted root} if and only if both $\alpha\neq0$ and $\fg_\alpha\neq0$. We denote the set of all restricted roots by $\Sigma\subset\fa^*$. As $\fg$ is finite dimensional, we obtain that $\Sigma$ is finite. We choose a connected component of $\fa\setminus\cup_{\alpha\in\Sigma}\ker(\alpha)$ and denote it by $\fa^{++}$. Also let $\fa^+$ denote the closure of $\fa^{++}$. Let $\sK\subset \sG$ (respectively $\sA\subset \sG$) be the connected subgroup whose Lie algebra is $\fk$ (respectively $\fa$) and let $\sA^+:=\exp{(\fa^+)}$. We note that $\sK$ is a maximal compact subgroup of $\sG$ (see Theorem 6.31 \cite{Knapp}).
\begin{theorem}
Let $\sG$ be a real semisimple Lie group of noncompact type. Then for any $g\in \sG$ there exists $k_1,k_2\in \sK$, not necessarily unique, and a unique $\kappa(g)\in \fa^+$ such that $g=k_1\exp{(\kappa(g))}k_2$ i.e.
$\sG=\sK\sA^+\sK$.
\end{theorem}
\begin{definition}
The map $\kappa:\sG\to\fa^+$ is called the Cartan projection.
\end{definition}
We note that the Cartan projection is invariant under conjugation only by elements of $\sK$. Let $\tB$ be the Killing form on $\fg$ i.e. for any $X,Y\in\fg$ we have $\tB(X,Y):=\ttr(\tad_X\circ \tad_Y)$ and let $\tB^*$ be another inner product on $\fg$ defined as follows: for all $X,Y\in\fg$ we have $\tB^*(X,Y):=-\tB(\theta(X),Y)$.
Then $\tB^*$ is positive definite and for $X\in\fk$ (respectively $\fp$) the operator $\tad_X$ is skew-symmetric (respectively symmetric) with respect to $\tB^*$ (for more details please see Lemma 6.44 of \cite{Zil}). Hence, we can choose a basis of $\fg$ with respect to which $\tB^*$ is the standard Euclidean norm and obtain an isomorphism $\iota: \sSL(\fg)\to \sSL(n,\R)$ for $n=\dim\fg$. Then under this isomorphsim $\tAd(\sK)$ is a subgroup of $\sSO(n,\R)$ and $\tAd(\sA)$ is a subgroup of $\sD$, the subgroup of diagonal matrices of $\sSL(n,\R)$. We observe that for an element $g\in \sSL(n,\R)$ its transpose, $g^t$ makes sense and for any $h\in \sSO(n,\R)$
we have $h^th=e$, where $e$ is the identity element of $\sSO(n,\R)$.

Now let $\sG$ be a connected semisimple real algebraic Lie group with trivial center and without compact factors. Since $\sG$ has trivial center, we note that $\iota\circ \tAd$ gives a rational embedding of $\sG$ into $\sSL(n,\R)$ (see Proposition 4.4.5 (ii) of \cite{Sp}). Henceforth, we will treat $\sG$ as an algebraic subgroup of $\sSL(n,\R)$. We consider the complexification, $\sSL(n,\C)$ of $\sSL(n,\R)$ and observe that $\sSL(n,\R)$ is Zariski dense in $\sSL(n,\C)$. Moreover, let $\sH_\C$ denote the Zariski closure of any algebraic subgroup $\sH$ of $\sSL(n,\R)$ inside $\sSL(n,\C)$. We consider the following regular embedding:
\begin{align*}
    \aleph: \sGL(n,\C)\times \sGL(n,\C) &\rightarrow \sGL(2n,\C)\\
    (g,h)&\mapsto   \begin{bmatrix} g & 0\\
                                    0 & h     
                    \end{bmatrix}.
\end{align*}
Hence $\aleph(\sG_\C\times\sG_\C)$ is a Zariski closed subgroup of $\sGL(2n,\C)$. Now for any $g_{11},g_{12},g_{21},g_{22}\in\mathfrak{gl}(n,\C)$ with $[g_{11},g_{12};g_{21},g_{22}]:=\begin{bmatrix} g_{11} & g_{12}\\ g_{21} & g_{22} \end{bmatrix}\in\sGL(2n,\C)$ we define $\cA:\sGL(2n,\C)\to\C$ such that \[\cA([g_{11},g_{12};g_{21},g_{22}]):=\ttr(g_{11}^tg_{11})-\ttr(g_{22}^tg_{22}).\]
We observe that $\cA$ is algebraic and we denote the zero set of $\cA$ inside $\sGL(2n,\C)$ by $\sZ_\cA$. 
\begin{lemma}\label{lem.nonempcar}
Let $\sG$ be a connected semisimple real Lie group with trivial center and without compact factors and let $\sN$ be any nontrivial normal subgroup of $\sG$. Then both $\aleph(\sN\times\{e\})\setminus\sZ_\cA$ and $\aleph(\{e\}\times\sN)\setminus\sZ_\cA$ are nonempty.
\end{lemma}
\begin{proof}
We use the arithmetic mean geometric mean inequality and observe that $\sZ_\cA\cap\aleph(\sG\times\{e\})=\aleph(\sK\times\{e\})$ and $\sZ_\cA\cap\aleph(\{e\}\times\sG)=\aleph(\{e\}\times\sK)$. Hence for any nontrivial normal subgroup $\sN$ of $\sG$ we have $\aleph(\sN\times\{e\})\cap\sZ_\cA=\aleph((\sN\cap\sK)\times\{e\})$ and similarly $\aleph(\{e\}\times\sN)\cap\sZ_\cA=\aleph(\{e\}\times(\sN\cap\sK))$.  

If possible, let us assume that $\aleph(\sN\times\{e\})\setminus\sZ_\cA=\emptyset$. It follows that $\aleph(\sN\times\{e\})\subset\sZ_\cA$ and we obtain that $\sN=\sN\cap\sK$. As $\sG$ have no compact factors and $\sN$ is nontrivial, we obtain a contradiction. Similarly, if we assume that $\aleph(\{e\}\times\sN)\setminus\sZ_\cA=\emptyset$, then we obtain a contradiction too. Therefore, we conclude that both $\aleph(\sN\times\{e\})\setminus\sZ_\cA\neq\emptyset$ and $\aleph(\{e\}\times\sN)\setminus\sZ_\cA\neq\emptyset$.
\end{proof}

\begin{lemma}\label{lem.eqcar}
Let $\sG$ be a connected semisimple real algebraic Lie group with trivial center and without compact factors. Then for $g,h\in\sG$ we have $\aleph(g,h)\in\sZ_\cA$, whenever $\kappa(g)=\kappa(h)$.
\end{lemma}
\begin{proof}
Let $g,h\in \sG$, let $g=k_g\exp{(\kappa(g))}k_g^\prime$ and $h=k_h\exp{(\kappa(h))}k_h^\prime$ for some $k_g,k_h,k_g^\prime,k_h^\prime\in \sK$. As $k^tk=e$ for any $k\in\sK$, we have $\ttr(k^tgk)=\ttr(g)$ for any $g\in\sG$. Hence we conclude by observing that for any $g,h\in\sG$ with $\kappa(g)=\kappa(h)$ we have
\begin{align*}
    \cA(\aleph(g,h))&=\ttr(g^tg)-\ttr(h^th)\\
    &=\ttr(\exp{(\kappa(g))}^t\exp{(\kappa(g))})-\ttr(\exp{(\kappa(h))}^t\exp{(\kappa(h))})=0.
\end{align*}
\end{proof}

\begin{theorem}\label{thm1}
Let $\sG$ be a connected semisimple real algebraic Lie group with trivial center and without compact factors. Let $\Gamma$ be a finitely generated group with a finite symmetric generating set $\sS$ and let $\cS$ be as in Proposition \ref{prop.linfinite}. Then, for $m=\max\{M_{\aleph(\sY_\C)\cap\sZ_\cA}\mid\sY\in\cS\}$ and any two representations $\rho_1,\rho_2: \Gamma \to \sG$, which satisfy the following:
\begin{enumerate}
    \item both $\rho_1(\Gamma)$ and $\rho_2(\Gamma)$ are Zariski dense in $\sG$,
    \item $\kappa(\rho_1(\gamma))=\kappa(\rho_2(\gamma))$ for all $\gamma \in \sB_\sS(m)$,
\end{enumerate}
are equivalent; in other words, there exists an automorphism $\sigma:\sG \to \sG$ such that we have $\sigma \circ \rho_1=\rho_2$. 
\end{theorem}

\begin{proof}
We use Proposition \ref{prop.linfinite} and observe that the Zariski closure of the group $(\rho_1,\rho_2)(\Gamma)$ inside $\sG\times\sG$ is
\[\sX:=\{(g,h)\mid h\sN_2=\sigma(g\sN_1)\}\]
for some normal subgroups $\sN_1,\sN_2$ of $\sG$ and isomorphism $\sigma:\sG/\sN_1\to\sG/\sN_2$. Let $\sX_\C$ be the Zariski closure of $\sX$ inside $\sGL(2n,\C)$. As $\aleph$ is a regular embedding, we observe that $\sZ_\cA\cap\aleph(\sX_\C)$ is a subvariety of $\sGL(2n,\C)$. Moreover, we have $M_{\aleph(\sX_\C)\cap\sZ_\cA}\leq m$. 

We use Lemma \ref{lem.eqcar} and observe that $\aleph(\rho_1(\gamma),\rho_2(\gamma))\in\sZ_\cA\cap\aleph(\sX_\C)$ for all $\gamma\in\sB_\sS(m)$. Hence $\aleph(\rho_1(\gamma),\rho_2(\gamma))\in\sZ_\cA\cap\aleph(\sX_\C)$ for all $\gamma\in\sB_\sS(M_{\aleph(\sX_\C)\cap\sZ_\cA})$. Therefore, by Proposition \ref{prop.main} we obtain that $\sZ_\cA\cap\aleph(\sX_\C)$ is not a proper subvariety of $\aleph(\sX_\C)$. It follows that $\aleph(\sX_\C)\subset\sZ_\cA$. Now we use Lemma \ref{lem.nonempcar} and deduce that $\sN_1$ and $\sN_2$ are trivial. Finally, we use Goursat's lemma \cite{Gour} to conclude the proof.
\end{proof}

\subsection{Jordan spectrum}

In this subsection we will prove finite step rigidity of Jordan spectrum. These results are partial generalizations of results of Dal'bo--Kim \cite{DK}. 

Let $\sG$ be a real semisimple Lie group of noncompact type and with trivial center. Also, let $\fg,\theta,\fk,\fp,\fa,\fa^{++},\fa^+,\fa^*$ and $\Sigma$ be as in Subsection \ref{subsec.carpro}. We define $\Sigma^+\subset \Sigma$ to be the set of restricted roots which take positive values on $\fa^+$ and consider the following nilpotent subalgebra:
\[\fn:=\bigoplus_{\alpha\in\Sigma^+}\fg_\alpha .\]
We denote the Lie subgroups of $\sG$ generated by $\fk,\fa,\fn$ as $\sK,\sA,\sN$ respectively and for $g\in \sG$ we say
\begin{enumerate}
    \item $g$ is \textit{elliptic} if and only if $g$ is conjugate to some element of $\sK$,
    \item $g$ is \textit{hyperbolic} if and only if $g$ is conjugate to some element of $\sA$,
    \item $g$ is \textit{unipotent} if and only if $g$ is conjugate to some element of $\sN$.
\end{enumerate}
\begin{theorem}[Jordan decomposition]
Let $\sG$ be a connected real semi-simple algebraic Lie group of noncompact type and let $g\in \sG$. Then there exist unique elliptic, hyperbolic and unipotent elements respectively $g_\te,g_\tth,g_\tu\in \sG$ such that the elements $g_\te,g_\tth,g_\tu$ commute with each other and
\[g=g_\te g_\tth g_\tu.\]
\end{theorem}
\begin{definition}\label{def.jordan}
Let $\sG$ be a connected semisimple real Lie group of noncompact type. Then the \textit{Jordan-Lyapunov projection} of $g\in \sG$, denoted by $\lambda_g$, is the unique element in $\fa^+$ such that $\exp{(\lambda_g)}$ is a conjugate of $g_\tth$.
\end{definition}
We observe that the Jordan projection are invariant under conjugation, i.e. for all $g,h\in \sG$ we have $\lambda_{hgh^{-1}}=\lambda_g$.
\begin{definition}
Let $\sG$ be a connected semisimple real Lie group of noncompact type. Then $g\in \sG$ is called \textit{loxodromic} if and only if $\lambda_g\in\fa^{++}$.
\end{definition}
Moreover, let $\sM$ be the centralizer of $\fa$ inside $\sK$ and let $\fm$ be the Lie algebra of $\sM$. We note that $\fg_0=\fm\oplus\fa$. 

Now let $\sG$ be a connected semisimple split real algebraic Lie group with trivial center and without compact factors. Since $\sG$ has trivial center, we observe that $\iota\circ \tAd$ gives a rational embedding of $\sG$ into $\sSL(n,\R)$ (see Proposition 4.4.5 (ii) of \cite{Sp}). Henceforth, we will treat $\sG$ as an algebraic subgroup of $\sSL(n,\R)$ and let $\aleph: \sSL(n,\C)\times \sSL(n,\C) \to \sGL(2n,\C)$ be as in Subsection \ref{subsec.carpro}. Hence $\aleph(\sG_\C\times\sG_\C)$ is a Zariski closed subgroup of $\sGL(2n,\C)$. Now for any $g_{11},g_{12},g_{21},g_{22}\in\mathfrak{gl}(n,\C)$ with $[g_{11},g_{12};g_{21},g_{22}]\in\sGL(2n,\C)$ we define $\cP:\sGL(2n,\C)\to\C$ such that \[\cP([g_{11},g_{12};g_{21},g_{22}]):=\ttr(g_{11}^2)-\ttr(g_{22}^2).\]
We observe that $\cP$ is algebraic and we denote the zero set of $\cP$ inside $\sGL(2n,\C)$ by $\sZ_\cP$. 

\begin{lemma}\label{lem.eqjor}
Let $\sG$ be a connected semisimple real algebraic Lie group with trivial center and without compact factors. Then for $g,h\in\sG$ with $g,h$ loxodromic, we have $\aleph(g,h)\in\sZ_\cP$, whenever $\lambda(g)=\lambda(h)$.
\end{lemma}
\begin{proof}
Let $g_1,g_2\in\sG$ and $g_1,g_2$ be loxodromic. Hence by Proposition 2.31 of \cite{Thi} there exists $h_1,h_2\in\sG$ and $m_1,m_2\in\sM$ such that $g_1=h_1m_1(\exp\lambda_{g_1})h_1^{-1}$ and $g_2=h_2m_2(\exp\lambda_{g_2})h_2^{-1}$. As $\sG$ is split, we use Theorem 7.5.3 of \cite{Knapp} and obtain that $m_i^2=e$ for both $i=1,2$. Hence we conclude by observing that for any loxodromic elements $g_1,g_2\in\sG$ with $\lambda(g_1)=\lambda(g_2)$ we have
\begin{align*}
    \cP(\aleph(g_1,g_2))&=\ttr(g_1^2)-\ttr(g_2^2)=\ttr(m_1^2(\exp\lambda_{g_1})^2)-\ttr(m_2^2(\exp\lambda_{g_2})^2)\\
    &=\ttr((\exp\lambda_{g_1})^2)-\ttr((\exp\lambda_{g_2})^2)=0.
\end{align*}
\end{proof}

\begin{lemma}\label{lem.nonempjor}
Let $\sG$ be a connected real split semisimple algebraic Lie group with trivial center and without compact factors. Let $\sN_1,\sN_2$ be two normal subgroups of $\sG$ such that there exists an isomorphism $\sigma:\sG/\sN_1\to\sG/\sN_2$ and let $\aleph(\sX)\subset\sZ_\cP$ where
\[\sX:=\{(g,h)\mid h\sN_2=\sigma(g\sN_1)\}.\]
Then both $\sN_1$ and $\sN_2$ are trivial.
\end{lemma}
\begin{proof}
We will prove our result via contradiction in two parts.

$\diamond$ If possible let us assume that $\sN_1$ is not trivial. Then there exists a semisimple element $g\in\sN_1$ such that $\ttr(g^2)\neq n$ (see Theorem of 4.3 \cite{Milne}). Hence $(g,e)\in\sX$ and \[\cP(\aleph(g,e))=\ttr(g^2)-\ttr(e)=\ttr(g^2)-n\neq0,\]
i.e. $\aleph(g,e)\notin\sZ_\cP$, a contradiction to the fact that $\aleph(\sX)\subset\sZ_\cP$.

$\diamond$ If possible let us assume that $\sN_2$ is not trivial. Then there exists a semisimple element $h\in\sN_2$ such that $\ttr(h^2)\neq n$. Hence $(e,h)\in\sX$ and \[\cP(\aleph(e,h))=\ttr(e)-\ttr(h^2)=n-\ttr(h^2)\neq0,\]
i.e. $\aleph(e,h)\notin\sZ_\cP$, a contradiction to the fact that $\aleph(\sX)\subset\sZ_\cP$.
\end{proof}

\begin{theorem}\label{thm2}
Let $\sG$ be a connected real split semisimple algebraic Lie group with trivial center and without compact factors. Let $\Gamma$ be a finitely generated group with a finite symmetric generating set $\sS \subset \Gamma$ and let $\cS$ be as in Proposition \ref{prop.linfinite}. Then, for $m=\max\{M_{\aleph(\sY_\C)\cap\sZ_\cP}\mid\sY\in\cS\}$ and two representations $\rho_1,\rho_2: \Gamma \to \sG$ which satisfy the following:
\begin{enumerate}
    \item $\rho_i(\Gamma)$ is Zariski dense in $\sG$,
    \item $\rho_i(\gamma)$ is loxodromic for all $\gamma\in\Gamma$,
    \item $\lambda(\rho_1(\gamma))=\lambda(\rho_2(\gamma))$ for all $\gamma \in \sB_\sS(m)$,
\end{enumerate}
are equivalent; in other words, there exists an automorphism $\sigma:\sG \to \sG$ such that we have $\sigma \circ \rho_1=\rho_2$.
\end{theorem}

\begin{proof}
We use Proposition \ref{prop.linfinite} and observe that the Zariski closure of the group $(\rho_1,\rho_2)(\Gamma)$ inside $\sG\times\sG$ is
\[\sX:=\{(g,h)\mid h\sN_2=\sigma(g\sN_1)\}\]
for some normal subgroups $\sN_1,\sN_2$ of $\sG$ and isomorphism $\sigma:\sG/\sN_1\to\sG/\sN_2$. Let $\sX_\C$ be the Zariski closure of $\sX$ inside $\sGL(2n,\C)$. As $\aleph$ is a regular embedding, we observe that $\sZ_\cP\cap\aleph(\sX_\C)$ is a subvariety of $\sGL(2n,\C)$. Moreover, we have $M_{\aleph(\sX_\C)\cap\sZ_\cP}\leq m$. 

We use Lemma \ref{lem.eqjor} and observe that $\aleph(\rho_1(\gamma),\rho_2(\gamma))\in\sZ_\cP\cap\aleph(\sX_\C)$ for all $\gamma\in\sB_\sS(m)$. Hence $\aleph(\rho_1(\gamma),\rho_2(\gamma))\in\sZ_\cP\cap\aleph(\sX_\C)$ for all $\gamma\in\sB_\sS(M_{\aleph(\sX_\C)\cap\sZ_\cP})$. Therefore, by Proposition \ref{prop.main} we obtain that $\sZ_\cP\cap\aleph(\sX_\C)$ is not a proper subvariety of $\aleph(\sX_\C)$. It follows that $\aleph(\sX_\C)\subset\sZ_\cP$. Now we use Lemma \ref{lem.nonempjor} and deduce that $\sN_1$ and $\sN_2$ are trivial. Finally, we use Goursat's lemma \cite{Gour} to conclude the proof.
\end{proof}

\begin{remark}
We note that if the condition of being Zariski dense is dropped from the assumptions then we get counterexamples. Indeed, let $\Gamma\subset\sSO(2,1)$ be a discrete free group with two generators and let $u:\Gamma\to\R^3$ be a map such that
\[u(\gamma\eta)=\gamma u(\eta)+u(\gamma)\]
for all $\gamma,\eta\in\Gamma$. We note that by \cite{Margulis1,Margulis2} there exists $u$ such that the Zariski closure of $\Gamma_u:=\{[\gamma,u(\gamma);0,1]\mid \gamma\in\Gamma\}\subset\sSL(4,\R)$ is
\[\sSO(2,1)\ltimes\R^3:=\{[g,v;0,1]\mid g\in\sSO(2,1), v\in\R^3\}\subset\sSL(4,\R).\]
Now we choose 
$\Gamma_0:=\{[\gamma,0;0,1]\mid \gamma\in\Gamma\}\subset\sSL(4,\R)$
and observe that the Zariski closure of $\Gamma_0$ is 
\[\{[g,0;0,1]\mid g\in\sSO(2,1)\}\subset\sSL(4,\R)\]
but $\lambda([\gamma,0;0,1])=\lambda([\gamma,u(\gamma);0,1])$ for all $\gamma\in\Gamma$.
\end{remark}

\subsection{Margulis-Smilga spectrum}\label{subsec.marglin}

In this subsection we will prove finite step rigidity of the Margulis-Smilga invariant spectrum for representations whose linear part is fixed.

Let $\sG$ be a real split semisimple algebraic Lie group of noncompact type with trivial center, $\sV$ be a finite dimensional vector space and let $\tR:\sG\to\mathsf{GL}(\sV)$ be a faithful irreducible algebraic representation. Let $\tdR:\fg\to\mathfrak{gl}(\sV)$ be the corresponding Lie algebra representation. We recall that the space of all linear forms on $\fa$ is denoted by $\fa^*$ and for all $\lambda\in\fa^*$ we consider
\[\sV^\lambda:=\{X\in\sV\mid \tdR_H(X)=\lambda(H)X\text{ for all } H\in\fa\}.\] 
We call $\lambda\in\fa^*$ a \textit{restricted weight} of the representation $\tR$ if and only if both $\lambda\neq0$ and $\sV^\lambda\neq0$. We denote the set of all nonzero restricted weights by $\Omega\subset\fa^*$. As $\sV$ is finite dimensional, we have $\Omega$ is a finite set. Moreover, we note that
\[\sV=\sV^0\oplus\bigoplus_{\lambda\in\Omega}\sV^\lambda.\]
We denote $\bigoplus_{\lambda\in\Omega}\sV^\lambda$ by $\sV^{\neq0}$. Let $\pi_0$ be the projection, with respect to the above decomposition, from $\sV$ onto $\sV^0$.

\begin{definition}\label{def.margi}
Let $(g,X)\in\GV$, let $g$ be loxodromic and let $g_\tth$ be the hyperbolic part of $g$ with respect to the Jordan decomposition. Also, let $h\in\sG$ be such that $g_\tth=h\texp(\tJd_g)h^{-1}$. Then $\tM(g,X)$, the \textit{Margulis-Smilga} invariant of $(g,X)$, is defined as follows:
\[\tM(g,X):=\pi_0(\tR_{h}^{-1}X).\]
\end{definition}

\begin{notation} 
Let $\tA\in\mathfrak{gl}(\sV)$. Hence $(\tI-\tA)\in\mathfrak{gl}(\sV)$. We define
\[\tP_\tA(x)=\sum_{k=0}^{\dim(\sV^{\neq0})}(-1)^{\dim(\sV^{\neq0})-k}\ttr(\wedge^{\dim(\sV^{\neq0})-k}(\tI-\tA))x^k.\]
\end{notation}
\begin{remark}\label{rem.poly}
Let $(g,X)\in\GV$, let $g$ be loxodromic and let $g_\tth$ be the hyperbolic part of $g$ with respect to the Jordan decomposition. Also, let $h\in\sG$ be such that $g_\tth=h\texp(\tJd_g)h^{-1}$. We use Propositions  3.9 and 5.5 of \cite{Ghosh5} to note that
\[\tP_g(\tR_e-\tR_g)X=\tP_g(0)\tR_h\tM(g,X).\]
\end{remark}
Let $n$ be the dimension of $\sV$ then without loss of generality we can consider $\sV=\R^n$. Also, as $\tR$ is algebraic and $\sG$ has no center, we obtain an embedding of $\sG\ltimes_\tR\sV$ inside $\sGL(n+1,\R)$ via the following map:
\[(g,X)\to\begin{bmatrix}\tR_g&X\\0&1\end{bmatrix}.\]
Henceforth, we will treat $\sG\ltimes_\tR\sV$ as an algebraic subgroup of $\sGL(n+1,\R)$ and denote the Zariski closure of $\sH:=\sG\ltimes_\tR\sV$ inside $\sGL(n+1,\C)$ by $\sH_\C$. Now for any $\tA\in\mathfrak{gl}(n,\C)$, $X,Y\in\C^n$, $k\in\R$ with $[\tA,X;Y^t,k]\in\sGL(n+1,\C)$ we define
$\cR:\sGL(n+1,\C)\to\C^n$ such that
\[\cR([\tA,X;Y^t,k]):=\tP_\tA(\tI-\tA)X.\]
We observe that $\cR$ is algebraic and we denote the zero set of $\cR$ inside $\sGL(n+1,\C)$ by $\sZ_\cR$. Now using Remark \ref{rem.poly}, it follows that for all loxodromic elements $g\in\sG$ and $X\in\sV$ with $\tM(g,X)=0$, we have $(\tR_g,X)\in\sZ_\cR$.

\begin{theorem}\label{thm3}
Let $\sG$ be a connected real split semisimple algebraic Lie group with trivial center and without compact factors. Let $\Gamma$ be a finitely generated group with a finite symmetric generating set $\sS \subset \Gamma$. Then, for $m=M_{\sH_\C\cap\sZ_\cR}$ and two representations $\rho_1,\rho_2: \Gamma \to \sH$ which satisfy the following:
\begin{enumerate}
    \item $\tL_{\rho_1}=g\tL_{\rho_2}g^{-1}$ for some $g\in\sG$,
    \item $\tL_{\rho_i}(\Gamma)$ is Zariski dense in $\sG$,
    \item $\tL_{\rho_i}(\gamma)$ is loxodromic for all $\gamma\in\Gamma$,
    \item $\tM(\rho_1(\gamma))=\tM(\rho_2(\gamma))$ for all $\gamma \in \sB_\sS(m)$,
\end{enumerate}
are equivalent; in other words, there exists $X\in\sV$ such that we have $(g,X)^{-1}\rho_1(g,X)=\rho_2$.
\end{theorem}
\begin{proof}
We consider $\rho:=(\tL_{\rho_1},\tT_{\rho_1}-\tR_g\tT_{\rho_2})$ and let $\sX$ be the Zariski closure of $\rho(\Gamma)$ inside $\sH$. 

We start by showing that $\sX$ is a proper subvariety of $\sH$. Indeed, if it is not then $\sX=\sH$. Also, we have that $\sH_\C\cap\sZ_\cR$ is a proper subvariety of $\sH_\C=\sX_\C$. Now using Proposition \ref{prop.main} we obtain that $\rho(\sB_\sS(M_{\sH_\C\cap\sZ_\cR}))\not\subset\sH_\C\cap\sZ_\cR$, a contradiction to our hypothesis which states that $\tM(\rho_1(\gamma))=\tM(\rho_2(\gamma))$ for all $\gamma \in \sB_\sS(M_{\sH_\C\cap\sZ_\cR})$ i.e. $\rho(\gamma)\in\sH_\C\cap\sZ_\cR$ for all $\gamma\in \sB_\sS(M_{\sH_\C\cap\sZ_\cR})$.

Now we consider the map $\tL:\sX\to\sG$ and its kernel $\ker(\left.\tL\right|_{\sX})\subset\{\tI\}\ltimes\sV$. As $\tL_\rho(\Gamma)$ is Zariski dense inside $\sG$, we deduce that $\left.\tL\right|_{\sX}$ is surjective. Moreover, for all $(\tA,X)\in\sH$ we have
\[(\tA,X)(\tI,V)(\tA,X)^{-1}=(\tI,\tA V).\]
Hence, we deduce that $\ker(\left.\tL\right|_{\sX})$ is a normal subgroup of $\sH$. Therefore, using Proposition A.2 of \cite{Ghosh5} and the fact that $\sX$ is a proper subvariety of $\sH$, we deduce that $\ker(\left.\tL\right|_{\sX})$ is trivial. Hence, $\left.\tL\right|_{\sX}$ is an isomorphism between $\sX$ and $\sG$. The inverse of this isomorphism induces a smooth map $X:\sG\to\sV$ such that $X_{gh}=\tR_gX_h+X_g$ for all $g,h\in\sG$ and $\sX=\{(h,X_h)\mid h\in\sG\}$. Now using Whitehead's Lemma (see page 13 of \cite{Raghu}), we conclude that there exists $X\in\sV$ such that $X_h=X-\tR_hX$ for all $h\in\sG$. We conclude that $\rho_1=(g,X)\rho_2(g,X)^{-1}$. 
\end{proof}

\subsection{Normed Margulis-Smilga spectrum}

In this subsection we will prove finite step rigidity of the normed Margulis-Smilga invariant spectrum for self-contragredient irreducible representations of split semisimple Lie groups.

Let $\sG, \Gamma, \tR, \tM, \tP, \sH$ be as in Subsection \ref{subsec.marglin}. Moreover, let $\tR$ be such that it admits an invariant symmetric bilinear form. We denote the norm on $\sV_\C$ coming from this symmetric bilinear form by $\|\cdot\|_\tR$. Now for any $\tA\in\mathfrak{gl}(n,\C)$, $X,Y\in\C^n$, $k\in\R$ with $[\tA,X;Y^t,k]\in\sGL(n+1,\C)$ we define
$\|\cR\|^2:\sGL(n+1,\C)\to\C$ such that
\[\|\cR\|^2([\tA,X;Y^t,k]):=\|\tP_\tA(\tI-\tA)X\|^2_\tR.\]
We observe that $\|\cR\|^2$ is algebraic and we denote the zero set of $\|\cR\|^2$ inside $\sGL(n+1,\C)$ by $\sZ_{\|\cR\|^2}$. Now using Remark \ref{rem.poly}, it follows that for all loxodromic elements $g\in\sG$ and $X\in\sV$ with $\|\tM(g,X)\|_\tR=0$, we have $(\tR_g,X)\in\sZ_{\|\cR\|^2}$.

Also, let $\aleph$ be as in Subsection \ref{subsec.carpro}. Hence, we get that $\aleph(\sH_\C\times\sH_\C)$ is a Zariski closed subgroup of $\sGL(2n+2,\C)$. Now for $i,j\in\{1,2\}$ and any $\tA_{ij}\in\mathfrak{gl}(n,\C)$, $X_{ij}, Y_{ij}\in\C^n$, $k_{ij}\in\C$ with $\tB_{ij}:=[\tA_{ij},X_{ij};Y_{ij}^t,k_{ij}]$ and $[\tB_{11},\tB_{12};\tB_{21},\tB_{22}]\in\sGL(2n+2,\C)$ we define $\cE:\sGL(2n+2,\C)\to\C$ such that
\[\cE\left(\begin{bmatrix}\tB_{11}&\tB_{12}\\\tB_{21}&\tB_{22}\end{bmatrix}\right):=\|\tP_{\tA_{11}}(0)\tP_{\tA_{22}}(\tI-\tA_{22})X_{22}\|^2_\tR-\|\tP_{\tA_{22}}(0)\tP_{\tA_{11}}(\tI-\tA_{11})X_{11}\|^2_\tR.\]
We observe that $\cE$ is algebraic and we denote the set of zeroes of $\cE$ inside $\sGL(2n+2,\C)$ by $\sZ_\cE$. Now using Remark \ref{rem.poly}, it follows that for all loxodromic elements $g,h\in\sG$ and $X,Y\in\sV$, with $\tM(g,X)=\tM(h,Y)$, we have $\aleph(g,X,h,Y)\in\sZ_\cE$.

\begin{theorem}\label{thm4}
Let $\sG$ be a connected real split semisimple algebraic Lie group with trivial center and without compact factors. Let $\Gamma$ be a finitely generated group with a finite symmetric generating set $\sS \subset \Gamma$ and let $\cS$ be as in Proposition \ref{prop.affinite}. Then, for $m=\max(\{M_{\aleph(\sY_\C)\cap\sZ_\cE}\mid\sY\in\cS\}\cup\{M_{\sH_\C\cap\sZ_{\|\cR\|^2}}\})$ and two representations $\rho_1,\rho_2: \Gamma \to \sH$ which satisfy the following:
\begin{enumerate}
    \item $\tL_{\rho_i}(\Gamma)$ is Zariski dense in $\sG$,
    \item $\tL_{\rho_i}(\gamma)$ is loxodromic for all $\gamma\in\Gamma$,
    \item $\|\tM(\rho_1(\gamma))\|_\tR=\|\tM(\rho_2(\gamma))\|_\tR$ for all $\gamma \in \sB_\sS(m)$,
\end{enumerate}
are either Zariski dense inside some conjugate of $\sG$ or there exists $(\tA,X)\in\sGL(\sV)\ltimes\sV$ such that we have $(\tA,X)^{-1}\rho_1(\tA,X)=\rho_2$.
\end{theorem}
\begin{proof}
We will prove this result in two parts.

$\diamond$ Let $\|\tM(\rho_1(\gamma))\|_\tR=0$ for all $\gamma\in\sB_\sS(m)$. We observe that $\sX_1$, the Zariski closure of $\rho_1(\Gamma)$ inside $\sH$, is a proper subvariety of $\sH$. Indeed, if it is not then $\sX_1=\sH$ and hence $\sH_\C\cap\sZ_{\|\cR\|^2}$ is a proper subvariety of $\sH_\C=(\sX_1)_\C$. As $M_{\sH_\C\cap\sZ_{\|\cR\|^2}}\leq m$, we use Proposition \ref{prop.main} and obtain a contradiction. 

Furthermore, as $\tL_{\rho_1}(\Gamma)$ is Zariski dense inside $\sG$, we deduce that $\left.\tL\right|_{\sX_1}$ is surjective. We also notice that $\ker(\left.\tL\right|_{\sX_1})\subset\{\tI\}\ltimes\sV$. Moreover, for all $(\tA,X)\in\sH$ we have
\[(\tA,X)(\tI,V)(\tA,X)^{-1}=(\tI,\tA V).\]
Hence, we deduce that $\ker(\left.\tL\right|_{\sX_1})$ is a normal subgroup of $\sH$. Therefore, using Proposition A.2 of \cite{Ghosh5} and the fact that $\sX_1$ is a proper subvariety of $\sH$, we deduce that $\ker(\left.\tL\right|_{\sX_1})$ is trivial. Hence, $\left.\tL\right|_{\sX_1}$ is an isomorphism between $\sX_1$ and $\sG$. The inverse of this isomorphism induces a smooth map $X:\sG\to\sV$ such that $X_{gh}=\tR_gX_h+X_g$ for all $g,h\in\sG$ and $\sX_1=\{(g,X_g)\mid g\in\sG\}$. Now using Whitehead's Lemma (see page 13 of \cite{Raghu}), we deduce that there exists $X\in\sV$ such that $\sX_1=(\tI,X)(\sG,0)(\tI,X)^{-1}$. Moreover, as $\|\tM(\rho_1(\gamma))\|_\tR=0$ for all $\gamma\in\sB_\sS(m)$, we obtain that $\|\tM(\rho_2(\gamma))\|_\tR=0$ for all $\gamma\in\sB_\sS(m)$. Hence, using a similar argument we deduce that $\sX_2$, the Zariski closure of $\rho_2(\Gamma)$ inside $\sH$, is conjugate to $\sG$.

$\diamond$ Let $\|\tM(\rho_1(\gamma))\|_\tR\neq0$ for some $\gamma\in\sB_\sS(m)$. It follows that $\|\tM(\rho_2(\gamma))\|_\tR\neq0$ for some $\gamma\in\sB_\sS(m)$. Hence, using Corollary 7.3 of \cite{Ghosh5} we obtain that both $\rho_1(\Gamma)$ and $\rho_2(\Gamma)$ are Zariski dense inside $\sG\ltimes_\tR\sV$. Let $\sX$ be the Zariski closure of $(\rho_1,\rho_2)(\Gamma)$ inside $(\sG\ltimes_\tR\sV)\times(\sG\ltimes_\tR\sV)$. Then by Proposition \ref{prop.affinite} there exist two normal subgroups $\sN_1,\sN_2$ of $\sG\ltimes_\tR\sV$ and a smooth isomorphism $\sigma:(\sG\ltimes_\tR\sV)/\sN_1\to (\sG\ltimes_\tR\sV)/\sN_2$ such that
\[\sX=\{(g,X,h,Y)\mid (h,Y)\sN_2=\sigma((g,X)\sN_1)\}.\]
We claim that either $\sN_1$ is trivial or $\sN_2$ is trivial. Indeed, if both $\sN_1$ and $\sN_2$ are non trivial, then using Lemma 9.2 of \cite{Ghosh5} we obtain that $\aleph(\sX)\setminus\sZ_\cE\neq\emptyset$. Hence, $\aleph(\sX_\C)\cap\sZ_\cE\neq\aleph(\sX_\C)$ i.e. $\aleph(\sX_\C)\cap\sZ_\cE$ is a proper subvariety of $\aleph(\sX_\C)$. Moreover, from our hypothesis it follows that $(\rho_1,\rho_2)(\sB_\sS(m))\subset\aleph(\sX_\C)\cap\sZ_\cE$. We also have that $M_{\aleph(\sX_\C)\cap\sZ_\cE}\leq m$ and we obtain a contradiction by using Proposition \ref{prop.main}. Therefore, we deduce that either $\sN_1$ is trivial or $\sN_2$ is trivial. Furthermore, we have
\begin{align*}
    \dim(\sN_1)&=\dim(\sG\ltimes_\tR\sV)-\dim((\sG\ltimes_\tR\sV)/\sN_1)\\
    &=\dim(\sG\ltimes_\tR\sV)-\dim(\sigma((\sG\ltimes_\tR\sV)/\sN_1))\\
    &=\dim(\sG\ltimes_\tR\sV)-\dim((\sG\ltimes_\tR\sV)/\sN_2)=\dim(\sN_2).
\end{align*}
As $\sG$ is semisimple without compact factors and with trivial center, we conclude that both $\sN_1$ and $\sN_2$ are trivial. Hence, $\sigma$ is a smooth automorphism of $\sG\ltimes_\tR\sV$ and we have
\[\sX=\{(g,X,h,Y)\mid(h,Y)=\sigma(g,X)\}.\]
Therefore, it follows that $\rho_2=\sigma\circ\rho_1$. Finally, we finish our proof by using Proposition 8.5 of \cite{Ghosh5}.
\end{proof}

\section{Applications}

In this section, we apply the results obtained in Section \ref{sec.fsr} and prove our main theorems about Anosov representations in split semisimple Lie groups with respect to minimal parabolic subgroups, Hitchin representations and Margulis-Smilga spacetimes.

\subsection{Anosov representations}

In this subsection we define the notion of an Anosov representation and show that for Anosov representations with respect to minimal parabolic subgroups inside split semisimple Lie groups, finite step rigidity hold for Jordan projection spectra. Anosov representations in $\sSL(n,\R)$ were introduced by Labourie in \cite{Labourie}. Later, in \cite{GW2}, Guichard--Wienhard extended the above definition and studied Anosov representations into real semisimple Lie groups. In recent years, important alternate characterizations of Anosov representations has been obtained in the works of Kapovich--Leeb--Porti \cite{KLPmain}, Bochi--Potrie--Sambarino \cite{BPS} and Potrie--Kassel \cite{KaPo}.

Let $\Gamma$ be a finitely generated word hyperbolic group and let $\bdry$ be its Gromov boundary. We note that $\Gamma$ has a natural action on $\bdry$. We consider
\[\bdry^{(2)}:=\{(x,y)\mid x,y\in\bdry, x\neq y\}\]
and denote $(\bdry^{(2)}\times\R)$ by $\cflow$. Moreover, for all $p_\pm\in\bdry$ and $p_0,t\in\R$, we define
\begin{align*}
    \phi_t:\cflow&\rightarrow\cflow\\
    p:=(p_-,p_+,p_0)&\mapsto (p_-,p_+,p_0+t).
\end{align*}
We recall a result of Gromov \cite{Gromov} to note that there exists a proper cocompact action of $\Gamma$ on $\cflow$ such that
\begin{enumerate}
    \item $\phi_t(\gamma p)=\gamma(\phi_tp)$ for all $t\in\R$ and $p\in\cflow$,
    \item for all $(p_-,p_+)\in\bdry^{(2)}$ and $\gamma\in\Gamma$ there exists $s_\gamma(p_-,p_+)\in\R$ such that $\gamma(p_-,p_+,0)=(\gamma p_-,\gamma p_+, s_\gamma(p_-,p_+)$. 
\end{enumerate}
Moreover, there exists a metric $d$ on $\cflow$, which is unique only up to H\"older equivalence, such that the action of $\Gamma$ is isometric with respect to $d$, the action of the flow is Lipschitz and every orbit of $\phi$ gives a quasi-isometric embedding. We denote the quotient of $\cflow$ under the action of $\Gamma$ by $\flow$. We abuse notation and denote the flow on $\flow$, coming from the flow $\phi$ on $\cflow$, by $\phi$.

Let $\sG$ be a real semisimple Lie group, let $\sP_\pm$ be a pair of opposite parabolic subgroups of $\sG$ and let
\[\cX:=\{(g\sP_-,g\sP_+)\mid g\in\sG\}\subset \sG/\sP_-\times\sG/\sP_+.\]
We note that $\cX$ is open inside $\sG/\sP_-\times\sG/\sP_+$ and it can be identified with $\sG/(\sP_-\cap\sP_+)$. Hence, we have
\[\sT_{(g\sP_-,g\sP_+)}\cX=\sT_{g\sP_-}\sG/\sP_-\oplus\sT_{g\sP_+}\sG/\sP_+.\]
\begin{definition}\label{def.anosov}
Let $\rho:\Gamma\to\sG$ be an injective homomorphism. Then $\rho$ is called Anosov with respect to $\sP_\pm$ if and only if
\begin{enumerate}
    \item there exist continuous, injective maps $\xi^\pm:\bdry\to\sG/\sP_\pm$ such that for all $\gamma\in\Gamma$, $x\in\bdry$ we have $\xi^\pm(\gamma x)=\rho(\gamma)\xi^\pm(x)$ and for all $p\in\cflow$ we have
    \[\xi(p):=(\xi^-(p_-),\xi^+(p_+))\in\cX,\]
    \item there exist positive constants $c,C$ and a continuous collection of Euclidean metrics
    \[\left\{\|\cdot\|_p:\sT_{\xi(p)}\cX\to\R\mid p\in\cflow\right\}\] such that for all $\gamma\in\Gamma$, $v\in\sT_{\xi(p)}\cX$ we have $\|\rho(\gamma)v\|_{\gamma p}=\|v\|_p$ and for all $v^\pm\in\sT_{\xi^\pm(p_\pm)}\sG/\sP_\pm$, $t\geq0$ we have
    \[\|v^\pm\|_{\phi_{\pm t}p}\leq C\exp(-ct)\|v^\pm\|_p.\]
\end{enumerate}
\end{definition}
\begin{theorem}\label{thm.ano}
Let $\sG$ be a connected real split semisimple algebraic Lie group with trivial center and without compact factors. Let $\Gamma$ be a finitely generated word hyperbolic group with a finite symmetric generating set $\sS \subset \Gamma$ and let $\cS$ be as in Proposition \ref{prop.linfinite}. Then, for $m=\max\{M_{\aleph(\sY_\C)\cap\sZ_\cP}\mid\sY\in\cS\}$ and two representations $\rho_1,\rho_2: \Gamma \to \sG$ which satisfy the following:
\begin{enumerate}
    \item $\rho_i(\Gamma)$ is Zariski dense in $\sG$,
    \item $\rho_i(\Gamma)$ is Anosov with respect to minimal parabolic subgroups,
    \item $\lambda(\rho_1(\gamma))=\lambda(\rho_2(\gamma))$ for all $\gamma \in \sB_\sS(m)$,
\end{enumerate}
are equivalent; in other words, there exists an automorphism $\sigma:\sG \to \sG$ such that we have $\sigma \circ \rho_1=\rho_2$.
\end{theorem}
\begin{proof}
We use Corollary 6.6 of \cite{KaPo} and Theorem \ref{thm2} to conclude our result.
\end{proof}

\subsection{Hitchin representations}

In this subsection we define the notion of a Hitchin representation and show that finite step rigidity of Jordan-Lyapunov spectra hold for these representations too.

\begin{definition}
Let $\Gamma$ be the fundamental group of a compact surface, let $\rho_0:\Gamma\to\sPSL(2,\R)$ be an injective homomorphism and let $\tF_n:\sPSL(2,\R)\to\sPSL(n,\R)$ be the irreducible representation. Then $\tF_n\circ\rho_0$ is called an $n$-Fuchsian representation.
\end{definition}
\begin{definition}
Let $\Gamma$ be the fundamental group of a compact surface and let $\sHom_{\sR}(\Gamma,\sPSL(n,\R))$ denote the space of reducible injective homomorphisms of $\Gamma$ inside $\sPSL(n,\R)$. Then connected components of the representation space $\sHom_{\sR}(\Gamma,\sPSL(n,\R))/\sPSL(n,\R)$ which contain $n$-Fuchsian representations are called Hitchin components. Any representation in a Hitchin component is called a Hitchin representation.
\end{definition}
\begin{remark}\label{rem.zclo}
Let $\Gamma$ be the fundamental group of a compact surface. We denote the space of Hitchin representations of $\Gamma$ in $\sPSL(n,\R)$ by $\cH_n(\Gamma)$. We note that by Lemma 10.1 of \cite{Labourie} all Hitchin representations are irreducible. Moreover, for $\rho_i\in\cH_n(\Gamma)$ we denote the Zariski closure of $\rho_i(\Gamma)$ inside $\sPSL(n,\R)$ by $\sG_i$ and the Zariski closure of $(\rho_i,\rho_j)(\Gamma)$ inside $\sPSL(n,\R)\times\sPSL(n,\R)$ by $\sG(\rho_i,\rho_j)$.
\end{remark}

\begin{theorem}\label{thm.hit1}
Let $\Gamma$ be the fundamental group of a compact surface, let $\sS\subset\Gamma$ be a finite symmetric generating set and let $\cS$ be as in Proposition \ref{prop.linfinite} for $\sG=\sPSL(n,\R)$. Then, for $m=\max\{M_{\aleph(\sY_\C)\cap\sZ_\cP}\mid\sY\in\cS\}$ and two representations $\rho_1,\rho_2\in\cH_n(\Gamma)$ which satisfy the following:
\begin{enumerate}
    \item $\rho_i(\Gamma)$ is Zariski dense in $\sPSL(n,\R)$,
    \item $\lambda(\rho_1(\gamma))=\lambda(\rho_2(\gamma))$ for all $\gamma \in \sB_\sS(m)$,
\end{enumerate}
are equivalent; in other words, there exists an automorphism $\sigma:\sG \to \sG$ such that we have $\sigma \circ \rho_1=\rho_2$.
\end{theorem}
\begin{proof}
Our result follows from Theorem 1.5 of \cite{Labourie} and Theorem \ref{thm2}.
\end{proof}

In fact, one can drop the Zariski density assumption from Theorem \ref{thm.hit1}. We will prove it below after the following result:
\begin{proposition}
Let $\Gamma$ be the fundamental group of a compact surface. Then 
$\max\{M_{\aleph(\sG(\rho,\eta)_\C)\cap\sZ_\cP}\mid \rho,\eta\in\cH_n(\Gamma)\}$ exists and is finite.
\end{proposition}
\begin{proof}
Let $\sG_\rho$, $\sG_\eta$ and $\sG(\rho,\eta)$ be as in Remark \ref{rem.zclo}. We use Lemma 2.19 of \cite{BCLS} and Theorem 1.5 of \cite{Labourie} to conclude that $\sG_\rho,\sG_\eta\subset\sPSL(n,\R)$ are semisimple Lie groups with trivial center. Moreover, by a recent result by Guichard \cite{Gui} we know that both $\sG_\rho$ and $\sG_\eta$ contain $\tF_n(\sPSL(2,\R))$ and hence there are finitely many options available for $\sG_\rho$ and $\sG_\eta$ (see also Corollary 1.5 of \cite{SambaZar}).

Let $\pi_i: \sG_\rho\times\sG_\eta\to\sG_i$ denote the natural projection map for $i\in\{\rho,\eta\}$. Hence, $\pi_i(\sG(\rho,\eta))=\sG_i$ for all $i\in\{\rho,\eta\}$. We denote the kernels of $\left.\pi_\rho\right|_{\sG(\rho,\eta)}$, $\left.\pi_\eta\right|_{\sG(\rho,\eta)}$ by $\sN_\rho$ and $\sN_\eta$ respectively. We observe that 
\begin{align*}
    \sN_\rho&=\sG(\rho,\eta)\cap(\{e\}\times\sG_\eta),\\
    \sN_\eta&=\sG(\rho,\eta)\cap(\sG_\rho\times\{e\}).
\end{align*}
It follows that $\sN_\rho$ is normal inside $\{e\}\times\sG_\eta$ and $\sN_\eta$ is normal inside $\sG_\rho\times\{e\}$. Hence, using Goursat's Lemma \cite{Gour} we obtain that the image of $\sG(\rho,\eta)$ inside $(\sG_\rho\times\sG_\eta)/\sN_\rho\sN_\eta$ is the graph of an isomorphism $\sigma:\sG_\rho/\sN_\eta\to\sG_\eta/\sN_\rho$. Moreover, as $\pi_\rho$, $\pi_\eta$ are smooth and $\sG(\rho,\eta)$ is an algebraic group, we deduce that $\sigma$ is smooth. As $\sG_\rho$ and $\sG_\eta$ are semisimple Lie groups with finite center and without compact factors (see Lemma 2.19 of \cite{BCLS}) we obtain that the number of possible options for $\sN_\rho$ and $\sN_\eta$ are finite. Furthermore, we note that any smooth isomorphism between two connected Lie groups is linear and any invertible linear transformation does not change the number of irreducible components, the degrees and the dimensions of an algebraic variety. Hence $\{M_{\aleph(\sG(\rho,\eta)_\C)\cap\sZ_\cP}\mid \rho,\eta\in\cH_n(\Gamma)\}$ is a finite set and we conclude that $\max\{M_{\aleph(\sG(\rho,\eta)_\C)\cap\sZ_\cP}\mid \rho,\eta\in\cH_n(\Gamma)\}$ exists and is finite.
\end{proof}

\begin{theorem}
Let $\Gamma$ be the fundamental group of a compact surface and let $\sS\subset\Gamma$ be a finite symmetric generating set. Then, for $m=\max\{M_{\sG(\rho,\eta)}\mid \rho,\eta\in\cH_n(\Gamma)\}$ and two representations $\rho_1,\rho_2\in\cH_n(\Gamma)$ which satisfy the following: $\lambda(\rho_1(\gamma))=\lambda(\rho_2(\gamma))$ for all $\gamma \in \sB_\sS(m)$, are equivalent; in other words, there exists an automorphism $\sigma:\sG \to \sG$ such that we have $\sigma \circ \rho_1=\rho_2$.
\end{theorem}
\begin{proof}
Let $\sG_\rho$, $\sG_\eta$ and $\sG(\rho,\eta)$ be as in Remark \ref{rem.zclo}. As $\aleph$ is a regular embedding, we observe that $\sZ_\cP\cap\aleph(\sG(\rho,\eta)_\C)$ is a subvariety of $\sGL(2n^2-2,\C)$. Moreover, we have $M_{\aleph(\sG(\rho,\eta)_\C)\cap\sZ_\cP}\leq m$. We use Lemma \ref{lem.eqjor} and observe that $\aleph(\rho(\gamma),\eta(\gamma))\in\sZ_\cP\cap\aleph(\sG(\rho,\eta)_\C)$ for all $\gamma\in\sB_\sS(m)$. Hence $\aleph(\rho(\gamma),\eta(\gamma))\in\sZ_\cP\cap\aleph(\sG(\rho,\eta)_\C)$ for all $\gamma\in\sB_\sS(M_{\aleph(\sG(\rho,\eta)_\C)\cap\sZ_\cP})$. Therefore, by Proposition \ref{prop.main} we obtain that $\sZ_\cP\cap\aleph(\sG(\rho,\eta)_\C)$ is not a proper subvariety of $\aleph(\sG(\rho,\eta)_\C)$. It follows that $\aleph(\sG(\rho,\eta)_\C)\subset\sZ_\cP$. 

We observe that $\pi_i(\sG(\rho,\eta)_\C)=(\sG_i)_\C$ and consider $\sN_i:=\ker(\left.\pi_i\right|_{\sG(\rho,\eta)_\C})$ for $i\in\{\rho,\eta\}$. It follows that $\sN_\rho$ is normal inside
$\{e\}\times(\sG_\eta)_\C$ and $\sN_\eta$ is normal inside
$\{e\}\times(\sG_\rho)_\C$. Moreover, we deduce that the image of $\sG(\rho,\eta)_\C$ inside $((\sG_\rho)_\C\times(\sG_\eta)_\C)/\sN_\rho\sN_\eta$ is the graph of a smooth isomorphism \[\sigma:(\sG_\rho)_\C/\sN_\eta\to(\sG_\eta)_\C/\sN_\rho.\] 

We claim that both $\sN_\rho$ and $\sN_\eta$ are trivial. Indeed, if not then either $\sN_\rho$ or $\sN_\eta$ is non trivial. Hence, either there exists a semisimple element $g\in\sN_\rho$ such that $\ttr(g^2)\neq(n^2-1)$ or there exists a semisimple element $h\in\sN_\eta$ such that $\ttr(h^2)\neq(n^2-1)$. It follows that either $\aleph(e,h)\notin\sZ_\cP$ or $\aleph(g,e)\notin\sZ_\cP$, a contradiction to the fact that $\aleph(\sG(\rho,\eta)_\C)\subset\sZ_\cP$.

Finally, we use Lemma 2.19 of \cite{BCLS} and Goursat's lemma \cite{Gour} to conclude the proof.
\end{proof}

\subsection{Margulis-Smilga spacetimes}

In this subsection we define the notion of a Margulis-Smilga spacetime and show that finite step rigidity of Margulis-Smilga invariant spectra hold for these representations when the Margulis-Smilga spacetime is coming from an irreducible representation of a split semisimple Lie group of noncompact type and with trivial center.

Let $\Gamma$ be a hyperbolic group, let $\sG$ be a semisimple Lie group of non compact type, let $\sV$ be a finite dimensional vector space, let $\tR:\sG\to\sGL(\sV)$ be a reducible representation. We denote the affine group coming from this representation by $\sG\ltimes_\tR\sV$.

\begin{definition}
Let $\rho:\Gamma\to\sG\ltimes_\tR\sV$ be an injective homomorphism such that $\rho(\gamma)$ is loxodromic for all non-identity element $\gamma\in\Gamma$. Then $\rho$ is called a Margulis-Smilga spacetime if and only if $\rho(\Gamma)$ acts properly on $\sV$.
\end{definition}

\begin{remark}\label{rem.nonzero}
We observe that, if $\tM(\rho(\gamma))=0$ for some $\gamma\in\Gamma$, then there exists $X_n\in\sV$ such that $X_n$ converge to some point $X\in\sV$ and $\rho(\gamma^n)X_n$ converge to some point $Y\in\sV$ as $n\to\infty$. Hence the Margulis-Smilga invariant spectrum of a Margulis-Smilga spacetime never vanishes.
\end{remark}

We denote the space of all Margulis-Smilga spacetimes as follows:
\[\cM(\Gamma,\sG\ltimes_\tR\sV):=\{\rho:\Gamma\to\sG\ltimes_\tR\sV\mid\rho(\Gamma) \text{ acts properly on } \sV \}/(\sG\ltimes_\tR\sV).\]
Moreover, we denote the subspace of $\cM(\Gamma,\sG\ltimes_\tR\sV)$, consisting of representations whose linear part is Zariski dense in $\sG$, by $\cM^{zdl}(\Gamma,\sG\ltimes_\tR\sV)$ and the subspace consisting of representations, whose linear part is inside the conjugacy class $[\varrho]$ for some representation $\varrho:\Gamma\to\sG$, by $\cM_{[\varrho]}(\Gamma,\sG\ltimes_\tR\sV)$.

\begin{theorem}
Let $\sG$ be a connected real split semisimple algebraic Lie group with trivial center and without compact factors. Let $\Gamma$ be a finitely generated group with a finite symmetric generating set $\sS \subset \Gamma$. Then, for $m=M_{\sH_\C\cap\sZ_\cR}$ and two representations $\rho_1,\rho_2\in\cM_{[\varrho]}(\Gamma,\sG\ltimes_\tR\sV)$ which satisfy the following:
\begin{enumerate}
    \item $\varrho(\Gamma)$ is Zariski dense in $\sG$,
    \item $\tM(\rho_1(\gamma))=\tM(\rho_2(\gamma))$ for all $\gamma \in \sB_\sS(m)$,
\end{enumerate}
are equivalent; in other words, there exists $X\in\sV$ such that we have $(g,X)^{-1}\rho_1(g,X)=\rho_2$.
\end{theorem}
\begin{proof}
Our result follows from Remark \ref{rem.nonzero} and Theorem \ref{thm3}.
\end{proof}

\begin{theorem}\label{thm.marg}
Let $\sG$ be a connected real split semisimple algebraic Lie group with trivial center and without compact factors, let $\tR$ preserves a norm $\|\cdot\|_\tR$ on $\sV$. Let $\Gamma$ be a finitely generated group with a finite symmetric generating set $\sS \subset \Gamma$ and let $\cS$ be as in Proposition \ref{prop.affinite}. Then, for $m=\max\{M_{\aleph(\sY_\C)\cap\sZ_\cE}\mid\sY\in\cS\}$ and two representations $\rho_1,\rho_2\in\cM^{zdl}(\Gamma,\sG\ltimes_\tR\sV)$ which satisfy the following: $\|\tM(\rho_1(\gamma))\|_\tR=\|\tM(\rho_2(\gamma))\|_\tR$ for all $\gamma \in \sB_\sS(m)$, are equivalent, i.e. there exists $(\tA,X)\in\sGL(\sV)\ltimes\sV$ such that we have $(\tA,X)^{-1}\rho_1(\tA,X)=\rho_2$.
\end{theorem}
\begin{proof}
Our result follows from Remark \ref{rem.nonzero} and Theorem \ref{thm4}.
\end{proof}
\begin{corollary}
Let $\sG$ be a connected real split semisimple algebraic Lie group with trivial center and without compact factors, let $\tR$ preserves a norm $\|\cdot\|_\tR$ on $\sV$. Let $\Gamma$ be a finitely generated group with a finite symmetric generating set $\sS \subset \Gamma$ and let $\cS$ be as in Proposition \ref{prop.affinite}. Then, for $m=\max\{M_{\aleph(\sY_\C)\cap\sZ_\cE}\mid\sY\in\cS\}$ and two representations $\rho_1,\rho_2\in\cM^{zdl}(\Gamma,\sG\ltimes_\tR\sV)$ which satisfy the following: $\tM(\rho_1(\gamma))=\tM(\rho_2(\gamma))$ for all $\gamma \in \sB_\sS(m)$, are equivalent, i.e. there exists $(\tA,X)\in\sGL(\sV)\ltimes\sV$ such that we have $(\tA,X)^{-1}\rho_1(\tA,X)=\rho_2$.
\end{corollary}
\begin{proof}
We observe that $\tM(\rho_1(\gamma))=\tM(\rho_2(\gamma))$ for all $\gamma \in \sB_\sS(m)$ imply that $\|\tM(\rho_1(\gamma))\|_\tR=\|\tM(\rho_2(\gamma))\|_\tR$ for all $\gamma \in \sB_\sS(m)$. Hence, our result follows from Theorem \ref{thm.marg}.
\end{proof}

\bibliography{Library.bib}

\newcommand{\etalchar}[1]{$^{#1}$}
\begin{thebibliography}{{Smi}16b}

\bibitem[AM59]{AM}
L.~Auslander and L.~Markus.
\newblock Flat {L}orentz {$3$}-manifolds.
\newblock {\em Mem. Amer. Math. Soc. No.}, 30:60, 1959.

\bibitem[AMS02]{AMS}
H.~Abels, G.~A. Margulis, and G.~A. Soifer.
\newblock On the {Z}ariski closure of the linear part of a properly
  discontinuous group of affine transformations.
\newblock {\em J. Differential Geom.}, 60(2):315--344, 2002.

\bibitem[AMS11]{AMS3}
H.~Abels, G.~A. Margulis, and G.~A. Soifer.
\newblock The linear part of an affine group acting properly discontinuously
  and leaving a quadratic form invariant.
\newblock {\em Geom. Dedicata}, 153:1--46, 2011.

\bibitem[AMS12]{AMS2}
Herbert {Abels}, Gregory {Margulis}, and Gregory {Soifer}.
\newblock {The Auslander conjecture for dimension less than 7}.
\newblock {\em arXiv e-prints}, page arXiv:1211.2525, Nov 2012.

\bibitem[Aus65]{Aus2}
Louis Auslander.
\newblock An account of the theory of crystallographic groups.
\newblock {\em Proceedings of the American Mathematical Society},
  16(6):1230--1236, 1965.

\bibitem[BBB{\etalchar{+}}17]{BBBKR}
Madeline Brandt, Juliette Bruce, Taylor Brysiewicz, Robert Krone, and Elina
  Robeva.
\newblock The degree of {${\rm SO}(n,\Bbb C)$}.
\newblock In {\em Combinatorial algebraic geometry}, volume~80 of {\em Fields
  Inst. Commun.}, pages 229--246. Fields Inst. Res. Math. Sci., Toronto, ON,
  2017.

\bibitem[BC17]{BC}
Martin Bridgeman and Richard~D. Canary.
\newblock Simple length rigidity for {K}leinian surface groups and
  applications.
\newblock {\em Comment. Math. Helv.}, 92(4):715--750, 2017.

\bibitem[BCL20]{BCL}
Martin Bridgeman, Richard Canary, and Fran\c{c}ois Labourie.
\newblock Simple length rigidity for {H}itchin representations.
\newblock {\em Adv. Math.}, 360:106901, 61, 2020.

\bibitem[BCLS15]{BCLS}
M.~Bridgeman, R.~Canary, F.~Labourie, and A.~Sambarino.
\newblock The pressure metric for {A}nosov representations.
\newblock {\em Geom. Funct. Anal.}, 25(4):1089--1179, 2015.

\bibitem[Bie11]{B1}
Ludwig Bieberbach.
\newblock \"{U}ber die {B}ewegungsgruppen der {E}uklidischen {R}\"{a}ume.
\newblock {\em Math. Ann.}, 70(3):297--336, 1911.

\bibitem[Bie12]{B2}
Ludwig Bieberbach.
\newblock \"{U}ber die {B}ewegungsgruppen der {E}uklidischen {R}\"{a}ume
  ({Z}weite {A}bhandlung.) {D}ie {G}ruppen mit einem endlichen
  {F}undamentalbereich.
\newblock {\em Math. Ann.}, 72(3):400--412, 1912.

\bibitem[BPS19]{BPS}
Jairo Bochi, Rafael Potrie, and Andr\'{e}s Sambarino.
\newblock Anosov representations and dominated splittings.
\newblock {\em J. Eur. Math. Soc. (JEMS)}, 21(11):3343--3414, 2019.

\bibitem[Bre11]{Br}
Emmanuel Breuillard.
\newblock A height gap theorem for finite subsets of {${\rm
  GL}_d(\overline{\Bbb Q})$} and nonamenable subgroups.
\newblock {\em Ann. of Math. (2)}, 174(2):1057--1110, 2011.

\bibitem[CD04]{CD}
V.~Charette and T.~Drumm.
\newblock Strong marked isospectrality of affine {L}orentzian groups.
\newblock {\em J. Differential Geom.}, 66(3):437--452, 03 2004.

\bibitem[CD10]{CDel}
Daryl Cooper and Kelly Delp.
\newblock The marked length spectrum of a projective manifold or orbifold.
\newblock {\em Proc. Amer. Math. Soc.}, 138(9):3361--3376, 2010.

\bibitem[CG93]{CG1}
Suhyoung Choi and William~M. Goldman.
\newblock Convex real projective structures on closed surfaces are closed.
\newblock {\em Proc. Amer. Math. Soc.}, 118(2):657--661, 1993.

\bibitem[Cro90]{Croke}
Christopher~B. Croke.
\newblock Rigidity for surfaces of nonpositive curvature.
\newblock {\em Comment. Math. Helv.}, 65(1):150--169, 1990.

\bibitem[Dan94]{Dl}
V.~I. Danilov.
\newblock Algebraic varieties and schemes.
\newblock In {\em Algebraic geometry, {I}}, volume~23 of {\em Encyclopaedia
  Math. Sci.}, pages 167--297. Springer, Berlin, 1994.

\bibitem[Dan19]{Thi}
Nguyen-Thi Dang.
\newblock {\em {Dynamics of group action on homogeneous spaces of higher rank
  and infinite volume}}.
\newblock Theses, {Universit{\'e} de Rennes 1}, September 2019.

\bibitem[DK02]{DK}
Fran\c{c}oise Dal'Bo and Inkang Kim.
\newblock Marked length rigidity for symmetric spaces.
\newblock {\em Comment. Math. Helv.}, 77(2):399--407, 2002.

\bibitem[DK15]{DeKe}
Harm Derksen and Gregor Kemper.
\newblock {\em Computational invariant theory}, volume 130 of {\em
  Encyclopaedia of Mathematical Sciences}.
\newblock Springer, Heidelberg, enlarged edition, 2015.
\newblock With two appendices by Vladimir L. Popov, and an addendum by Norbert
  A'Campo and Popov, Invariant Theory and Algebraic Transformation Groups,
  VIII.

\bibitem[DLR10]{DLR}
Moon Duchin, Christopher~J. Leininger, and Kasra Rafi.
\newblock Length spectra and degeneration of flat metrics.
\newblock {\em Invent. Math.}, 182(2):231--277, 2010.

\bibitem[EMO05]{EMO}
Alex Eskin, Shahar Mozes, and Hee Oh.
\newblock On uniform exponential growth for linear groups.
\newblock {\em Invent. Math.}, 160(1):1--30, 2005.

\bibitem[FG83]{FG}
D.~Fried and W.~M. Goldman.
\newblock Three-dimensional affine crystallographic groups.
\newblock {\em Adv. in Math.}, 47(1):1--49, 1983.

\bibitem[Ful84]{Fulton}
William Fulton.
\newblock {\em Introduction to intersection theory in algebraic geometry},
  volume~54 of {\em CBMS Regional Conference Series in Mathematics}.
\newblock Published for the Conference Board of the Mathematical Sciences,
  Washington, DC; by the American Mathematical Society, Providence, RI, 1984.

\bibitem[{Gho}19]{Ghosh4}
Sourav {Ghosh}.
\newblock {Margulis Multiverse: Infinitesimal Rigidity, Pressure Form and
  Convexity}.
\newblock {\em arXiv e-prints}, page arXiv:1907.12348, Jul 2019.

\bibitem[{Gho}20]{Ghosh5}
Sourav {Ghosh}.
\newblock {Isospectrality of Margulis-Smilga spacetimes for irreducible
  representations of real split semisimple Lie groups}.
\newblock {\em arXiv e-prints}, page arXiv:2009.12746, September 2020.

\bibitem[Gol88]{Goldman}
William~M. Goldman.
\newblock Topological components of spaces of representations.
\newblock {\em Invent. Math.}, 93(3):557--607, 1988.

\bibitem[Gou89]{Gour}
Edouard Goursat.
\newblock Sur les substitutions orthogonales et les divisions r\'{e}guli\`eres
  de l'espace.
\newblock {\em Ann. Sci. \'{E}cole Norm. Sup. (3)}, 6:9--102, 1889.

\bibitem[Gro87]{Gromov}
M.~Gromov.
\newblock Hyperbolic groups.
\newblock In {\em Essays in group theory}, volume~8 of {\em Math. Sci. Res.
  Inst. Publ.}, pages 75--263. Springer, New York, 1987.

\bibitem[Gui]{Gui}
Olivier Guichard.
\newblock Zariski closure of positive and maximal representations.

\bibitem[Gui08]{Gui1}
Olivier Guichard.
\newblock Composantes de {H}itchin et repr\'{e}sentations hyperconvexes de
  groupes de surface.
\newblock {\em J. Differential Geom.}, 80(3):391--431, 2008.

\bibitem[GW12]{GW2}
O.~Guichard and A.~Wienhard.
\newblock Anosov representations: domains of discontinuity and applications.
\newblock {\em Invent. Math.}, 190(2):357--438, 2012.

\bibitem[Hit92]{Hit}
N.~J. Hitchin.
\newblock Lie groups and {T}eichm\"{u}ller space.
\newblock {\em Topology}, 31(3):449--473, 1992.

\bibitem[Kaz87]{Kazar}
B.~Ya. Kazarnovski\u{\i}.
\newblock Newton polyhedra and {B}ezout's formula for matrix functions of
  finite-dimensional representations.
\newblock {\em Funktsional. Anal. i Prilozhen.}, 21(4):73--74, 1987.

\bibitem[Kim01]{Kim2}
Inkang Kim.
\newblock Rigidity and deformation spaces of strictly convex real projective
  structures on compact manifolds.
\newblock {\em J. Differential Geom.}, 58(2):189--218, 2001.

\bibitem[Kim05]{Kim}
I.~Kim.
\newblock Affine action and margulis invariant.
\newblock {\em Journal of Functional Analysis}, 219(1):205--225, 2005.

\bibitem[KLP18]{KLPmain}
Michael Kapovich, Bernhard Leeb, and Joan Porti.
\newblock A {M}orse lemma for quasigeodesics in symmetric spaces and
  {E}uclidean buildings.
\newblock {\em Geom. Topol.}, 22(7):3827--3923, 2018.

\bibitem[Kna02]{Knapp}
Anthony~W. Knapp.
\newblock {\em Lie groups beyond an introduction}, volume 140 of {\em Progress
  in Mathematics}.
\newblock Birkh\"{a}user Boston, Inc., Boston, MA, second edition, 2002.

\bibitem[KP20]{KaPo}
Fanny {Kassel} and Rafael {Potrie}.
\newblock {Eigenvalue gaps for hyperbolic groups and semigroups}.
\newblock {\em arXiv e-prints}, page arXiv:2002.07015, February 2020.

\bibitem[Lab06]{Labourie}
F.~Labourie.
\newblock Anosov flows, surface groups and curves in projective space.
\newblock {\em Invent. Math.}, 165(1):51--114, 2006.

\bibitem[Mar83]{Margulis1}
G.~A. Margulis.
\newblock Free completely discontinuous groups of affine transformations.
\newblock {\em Dokl. Akad. Nauk SSSR}, 272(4):785--788, 1983.

\bibitem[Mar84]{Margulis2}
G.~A. Margulis.
\newblock Complete affine locally flat manifolds with a free fundamental group.
\newblock {\em Zap. Nauchn. Sem. Leningrad. Otdel. Mat. Inst. Steklov. (LOMI)},
  134:190--205, 1984.
\newblock Automorphic functions and number theory, II.

\bibitem[Mil77]{Mil}
J.~Milnor.
\newblock On fundamental groups of complete affinely flat manifolds.
\newblock {\em Advances in Math.}, 25(2):178--187, 1977.

\bibitem[{Mil}07]{Milne}
J.~S. {Milne}.
\newblock {Semisimple Algebraic Groups in Characteristic Zero}.
\newblock {\em arXiv e-prints}, page arXiv:0705.1348, May 2007.

\bibitem[Ota90]{Otal}
Jean-Pierre Otal.
\newblock Le spectre marqu\'{e} des longueurs des surfaces \`a courbure
  n\'{e}gative.
\newblock {\em Ann. of Math. (2)}, 131(1):151--162, 1990.

\bibitem[Rag07]{Raghu}
M.~S. Raghunathan.
\newblock Discrete subgroups of {L}ie groups.
\newblock {\em Math. Student}, Special Centenary Volume:59--70 (2008), 2007.

\bibitem[{Sam}20]{SambaZar}
Andr{\'e}s {Sambarino}.
\newblock {Infinitesimal Zariski closures of positive representations}.
\newblock {\em arXiv e-prints}, page arXiv:2012.10276, December 2020.

\bibitem[Sch00]{Sch}
A.~Schinzel.
\newblock {\em Polynomials with special regard to reducibility}, volume~77 of
  {\em Encyclopedia of Mathematics and its Applications}.
\newblock Cambridge University Press, Cambridge, 2000.
\newblock With an appendix by Umberto Zannier.

\bibitem[Smi16a]{Smilga}
I.~Smilga.
\newblock Proper affine actions on semisimple {L}ie algebras.
\newblock {\em Ann. Inst. Fourier (Grenoble)}, 66(2):785--831, 2016.

\bibitem[{Smi}16b]{Smilga4}
Ilia {Smilga}.
\newblock {Proper affine actions: a sufficient criterion}.
\newblock {\em arXiv e-prints}, page arXiv:1612.08942, December 2016.

\bibitem[Smi18]{Smilga3}
Ilia Smilga.
\newblock Proper affine actions in non-swinging representations.
\newblock {\em Groups Geom. Dyn.}, 12(2):449--528, 2018.

\bibitem[Spr09]{Sp}
T.~A. Springer.
\newblock {\em Linear algebraic groups}.
\newblock Modern Birkh\"{a}user Classics. Birkh\"{a}user Boston, Inc., Boston,
  MA, second edition, 2009.

\bibitem[{Zil}10]{Zil}
Wolfgang {Ziller}.
\newblock Lie groups. representation theory and symmetric spaces.
\newblock \url{https://www.math.upenn.edu/~wziller/math650/LieGroupsReps.pdf},
  2010.
\newblock [Online; accessed 22-June-2021].

\end{thebibliography}
\bibliographystyle{alpha}

\end{document}